\pdfoutput=1
\documentclass[11pt,a4paper]{amsart}  
\usepackage[utf8]{inputenc}
\usepackage[T1]{fontenc}
\usepackage{mathtools}
\usepackage{amssymb}
\mathtoolsset{showonlyrefs=true} 
\usepackage{times}  
\usepackage{xspace}
\usepackage[final]{microtype} 
\usepackage{paralist} 
\usepackage{enumitem}
\usepackage[pdfusetitle,colorlinks,bookmarksopen=true,bookmarksnumbered=true,linkcolor=blue,urlcolor=blue,citecolor=blue]{hyperref}
\usepackage{lineno}
\usepackage{pgfplots}
\pgfplotsset{compat=newest}
\usetikzlibrary{patterns}
\usepgfplotslibrary{fillbetween}
\usepackage{tikzscale}
\usepackage{colortbl}
\usepackage{booktabs}
\usepackage{algorithm}
\usepackage[noend]{algpseudocode}
\usepackage[textsize=tiny,shadow]{todonotes}
\usepackage{ stmaryrd }
\usepackage{setspace}
\newtheorem{theorem}{Theorem}[section]
\newtheorem{lemma}[theorem]{Lemma}

\newtheorem{proposition}[theorem]{Proposition}
\newtheorem{corollary}[theorem]{Corollary}

\theoremstyle{definition}
\newtheorem{definition}[theorem]{Definition}

\newtheorem{remark}[theorem]{Remark}

\DeclareMathOperator{\dist}{\mathrm{dist}}

\newcommand{\eps}{\varepsilon}

\newcommand{\NN}{\mathbb{N}}
\newcommand{\RR}{\mathbb{R}}
\newcommand{\CC}{\mathbb{C}}
\newcommand{\ZZ}{\mathbb{Z}}
\newcommand{\II}{\mathbb{I}}
\newcommand{\JJ}{\mathbb{J}}

\newcommand{\im}{{\rm im}}
\let\subsec\S
\renewcommand{\S}{\Sigma}
\newcommand{\W}{{\mathcal W}}
\newcommand{\C}{{\mathcal C}}

\DeclareMathOperator{\supp}{supp}
\DeclareMathOperator{\diam}{diam}

\hypersetup{
pdfcreator={LaTeX with hyperref (master, b3d7486)},
pdfkeywords={Finite section method, limit operators, Jacobi operators},
pdfsubject={MSC Primary 65J10, 47B36; Secondary 47N50},
}

\newcommand{\nofigure}[2]{
\smallskip
\noindent
\refstepcounter{figure} 
\includegraphics[width=\textwidth]{#1}
\begin{spacing}{0.8}
{\footnotesize \noindent {\bf Figure \thefigure.} 
#2\\
}
\end{spacing}
}

\newcommand{\nonofigure}[3]{
\smallskip
\noindent
\refstepcounter{figure} 
\begin{tabular}{lr}
\includegraphics[width=0.46\textwidth]{#1}&
\includegraphics[width=0.46\textwidth]{#2}\\
\end{tabular}
\begin{spacing}{0.8}
{\footnotesize \noindent {\bf Figure \thefigure.} 
#3\\
}
\end{spacing}
}

\begin{document}
\title[Spectral approximation via approximation of subwords]{Spectral approximation\\of generalized Schrödinger operators\\via approximation of subwords}

\author[Gabel, Gallaun, Gro{\ss}mann, Lindner, Ukena]{Fabian Gabel,\, Dennis Gallaun,\, Julian Gro{\ss}mann,\\Marko Lindner,\, Riko~Ukena}
\address{
	Hamburg University of Technology,
	Institute of Mathematics,
	Am Schwarzen\-berg-Campus 3,
	D-21073 Hamburg,
	Germany}
\email{\href{mailto:fabian.gabel@tuhh.de}{fabian.gabel@tuhh.de}}
\email{\href{mailto:dennis.gallaun@tuhh.de}{dennis.gallaun@tuhh.de}}
\email{\href{mailto:julian.grossmann@jp-g.de}{julian.grossmann@jp-g.de}}
\email{\href{mailto:marko.lindner@tuhh.de}{marko.lindner@tuhh.de}}
\email{\href{mailto:riko.ukena@tuhh.de}{riko.ukena@tuhh.de}}


\begin{abstract}
We demonstrate criteria, purely based on finite subwords of the potential, to guarantee spectral inclusion as well as Hausdorff approximation of pseudospectra or even spectra of generalized Schrödinger operators on the discrete line or half-line. In fact, our results are neither limited to Schrödinger or self-adjoint operators, nor to Hilbert space or 1D.
\end{abstract}

\maketitle

\section{Introduction}\label{sec:intro}
\subsection*{The operators and their matrices}
Discretizing the (self-adjoint) 1D Schrödinger operator $\Delta+ v\cdot$ by finite differences, one gets an infinite tridiagonal matrix of the form $A$ below, where the sequence $(b_k)_{k\in\ZZ}$ comes from samples of the function $v$.

When we speak of \emph{generalized Schrödinger operators} here, we think of bounded linear operators on $\ell^2(\ZZ)$ with matrix representation of the form $B$, where
\[
A=\begin{pmatrix}
\ddots & \ddots\\
\smash\ddots & b_{-1} & 1 \\
& 1 & b_0 & 1\\
&& 1 & b_1 & 1\\
&&& 1& b_2&\smash\ddots\\
&&&&\ddots &\ddots
\end{pmatrix},
\quad
B=\begin{pmatrix}
\ddots&\ddots\\
\smash\ddots& *&\bullet\\
\smash\ddots& b_{-1}&*&\bullet\\
& \star &b_0&*&\bullet\\
&& \star &b_1&*&\smash\ddots\\
&&&\ddots&\ddots&\ddots
\end{pmatrix}
\]
and $\star$, $*$ and $\bullet$ mark constant (but possibly different) diagonals.
The number of nonzero diagonals is not important, as long as it is finite, and the one varying diagonal, carrying the so-called \emph{potential} $b=(b_k)_{k\in\ZZ}$, could be anywhere. In Section \ref{sec:ext} below we explain how (and how far) also these limitations can be overcome. 
We denote the generalized Schrödinger operator with matrix $B$ by $H(b)$.

\subsection*{Our results}
The aim of this paper is to demonstrate how the set $\W(b)$ of all finite \emph{subwords} (vectors of consecutive entries, a.k.a.~\emph{factors}) of $b$ determines spectral quantities of $H(b)$, such as the spectrum, the pseudospectra and the smallest singular value. 
We compare these spectral quantities between two operators, $H(b)$ and $H(c)$, and prove their approximation by those of a sequence, $H(b_m)$, purely by looking at the sets $\W(b)$, $\W(c)$ and $\W(b_m)$. More precisely:
\newpage

\begin{itemize}
\item
If $\W(b)=\W(c)$, then all aforementioned spectral quantities of $H(b)$ and $H(c)$ coincide, irrelevant at which position a subword $w\in\W(b)$ appears in $c$. 
For example, $H(b)$ with a random $\{0,1\}$-potential has the same spectrum as $H(c)$ with $c\in\{0,1\}^\ZZ$ listing all natural numbers in binary form.

\item
We explicitly compare the same quantities of one band operator, between the axis and the half-axis (with homogeneous Dirichlet boundary condition).

\item
We approximate the pseudospectrum (and in the normal case also the spectrum) of $H(b)$ in the Hausdorff distance by the pseudospectra of $H(b_m)$, where $b_m$ and $b$ have the same subwords of length $N$ and where $N=N(m)\to\infty$ if $m\to\infty$.

\item
For the corresponding approximation result on the half-axis, we have to assume, in addition, that $b_m\to b$ pointwise.
\end{itemize}

We demonstrate our approximations of spectra and pseudospectra for different configurations including the self-adjoint and a non-self-adjoint Fibonacci Hamiltonian, 
the non-self-adjoint Anderson model, Feinberg \& Zee's randomly hopping particle and a one-way version of it with two varying diagonals. 
The potentials in these examples are aperiodic or pseudoergodic~\cite{Gabel.2023}.
See Sections~\ref{sec:classes} and~\ref{sec:ex_math_phys} and Figure~\ref{fig:intro} for the details.

Our results might not be too surprising for experts on self-adjoint Schrödinger operators $H(b)=A$ from above, see~\cite{Beckus.2020, Kellendonk.2019},
but our tools are such that everything works as well for the generalized version $H(b)=B$, also with more than one varying diagonal, 
and in $\ell^p(\ZZ^d)$ with entries in a Banach space $X$, see Section~\ref{sec:ext}.

\nonofigure{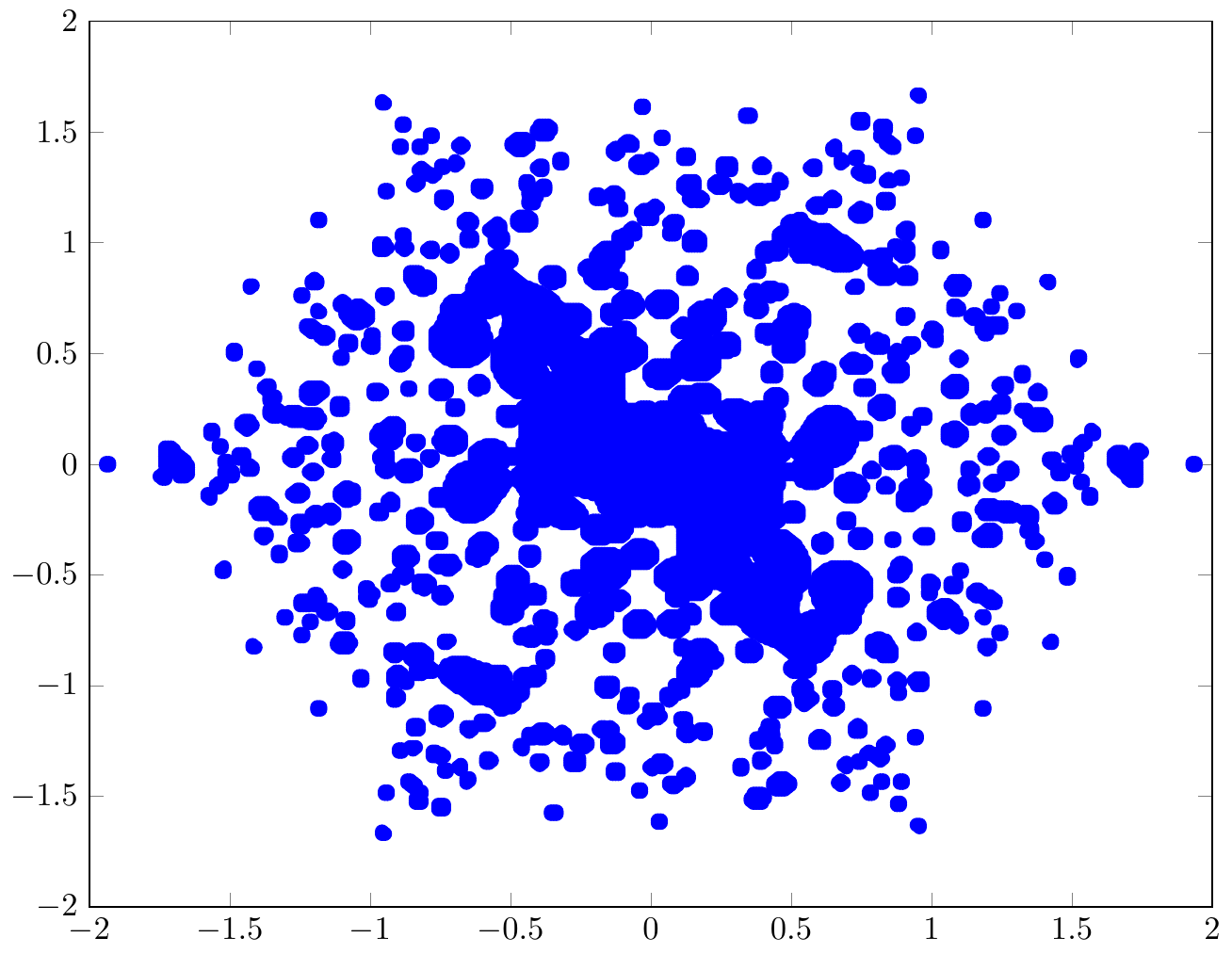}{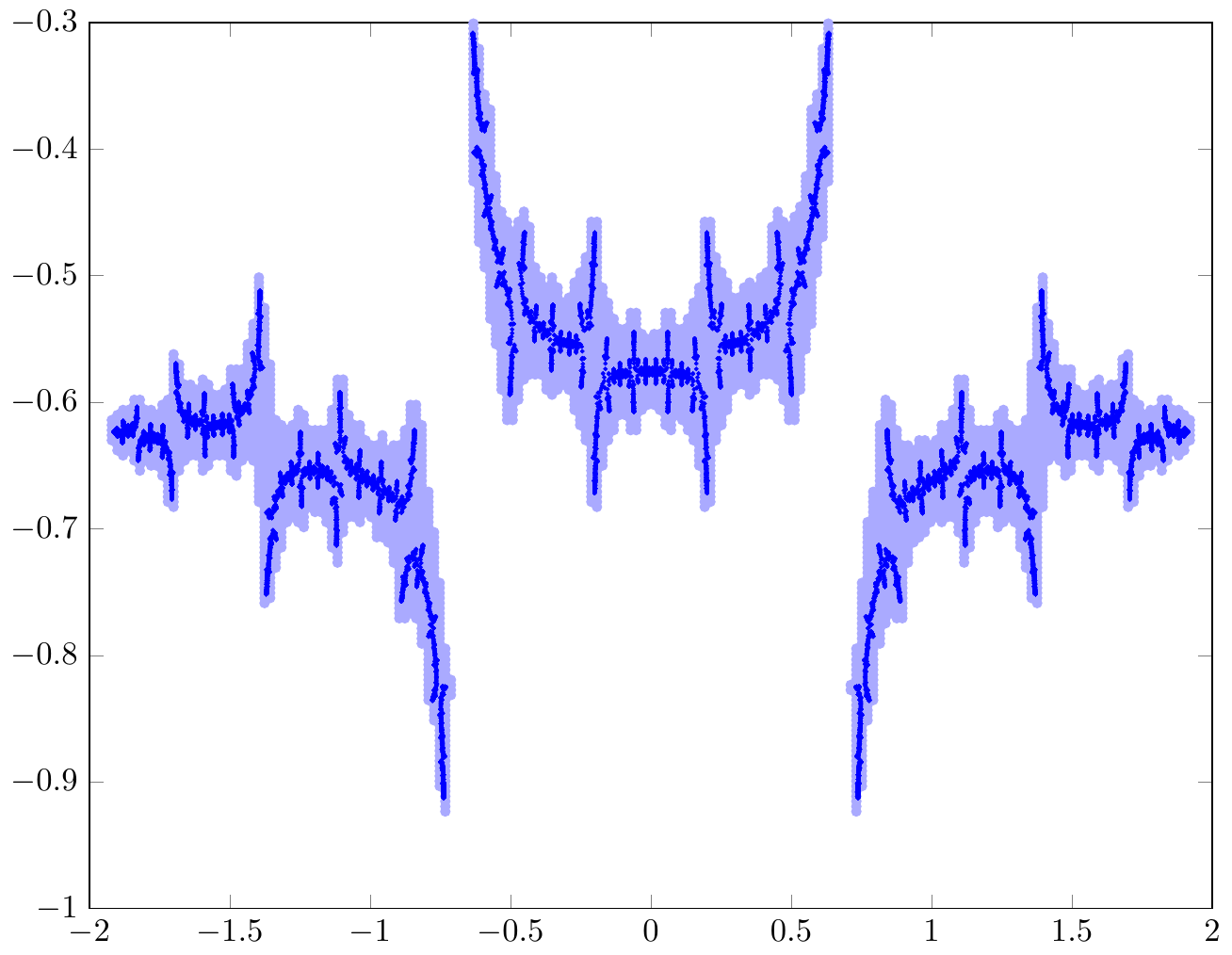}
{Pseudospectral approximations of the 3-state randomly hopping particle \cite{FeinZee97,CicutaContediniMolinari2000} 
and a non-self-adjoint version of the Fibonacci Hamiltonian~\cite{Damanik.2016}; 
see \subsec\ref{sec:ex_math_phys} for details and other operators. \label{fig:intro}}
~\\[-16mm]  

\subsection*{Our message.}
Maybe our message is that spectral approximation is about approximating $\|(A-\lambda)^{-1}\|$, not necessarily $(A-\lambda)^{-1}$. 
In particular, it is not about cutting larger and larger portions out of the operator, 
but instead about capturing the variety of its finite subpatterns: Finding all the $2^N$ subwords of length $N$ in a random $\{0,1\}$-potential $b$ can take unpredictably large sections and create huge numerical costs. Arranging all of them in a new ``surrogate'' potential $c$ takes $N\cdot 2^N$ letters if arranged naively, and just $2^N$ if in a clever, condensed arrangement \cite{deBruijn}.

So if a model $H(b)$ allows a priori knowledge of the set $\W(b)$ of finite subwords then we suggest to build a surrogate potential 
that has all those subwords of length $N$ (but no others), maybe neatly overlapping, and to periodize it.
Doing this for an increasing sequence $N=N_1, N_2,\dots$, this will Hausdorff-approximate the pseudospectrum of $H(b)$ -- in normal models also the spectrum. 
Spectra of periodic band operators can be explicitly computed by so-called Floquet-Bloch theory, see e.g.~\cite[Thm.~4.4.9]{Davies.2007} (two-sided case) and \cite[Thm.~4.42]{HaggerPhD} (one-sided case).

\subsection*{Related literature.} 
Always a great inspiration is \cite{Davies2001:PseudoErg} by Brian Davies. 
The way he expresses one essence of random behavior via pseudoergodicity -- 
maximal variety of subwords -- is guiding our view on spectral theory, 
not only here and not only for random operators. 
Aperiodic models are also characterized by their subword variety but now it is minimal (among the nontrivial classes). Spectral theory of aperiodic operators, see e.g.~ \cite{Suto.1987,Bellissard.1989,DamanikLenz2006,Stollmann,Damanik.2015,Damanik.2016}, was actually our starting point for this research.
We learned about asymptotics of pseudospectra, in particular in the Hausdorff topology, 
from Böttcher, Hagen, Roch and Silbermann \cite{BoeSi2,Hagen.2001}
 -- going from concrete results for Toeplitz and related operators to more generic statements for band-dominated operators.
One of our main tools, reducing spectral studies to finite column sets -- a.k.a.~one-sided or rectangular finite sections -- 
is the localization of the lower norm (Lemma~\ref{lem:nuN}), developed in \cite{CW.Heng.ML:UpperBounds,BigQuest} and extended in~\cite{LiSchmidt:Givens}.

In parallel, Ben-Artzi, Colbrook, Hansen, Nevanlinna, Roman and Seidel \cite{SCI,ColbRomanHansen:PRL} also looked at spectral approximations via rectangular submatrices but combined with so-called $(N,\eps)$-pseudospectra \cite{Hansen:nPseudo,Seidel:Neps}. Of particular relevance are their complexity studies for many computational, including spectral, problems.

Of course, this list is incomplete; many other groups are studying spectral quantities of non-self-adjoint operators, e.g.~\cite{Boulton08,Boegli2017,Boegli2020,Jacob21}, in terms of pseudospectra but also numerical ranges, higher order spectra, polynomial hulls, etc.

\subsection*{Structure of the paper}
After introducing the main actors, our philosophy and the results in this introduction, we give proper definitions of the operator classes and discuss concrete examples from mathematical physics in Section \ref{sec:examples}. Sections \ref{sec:notations} and \ref{sec:results} introduce our main tools and techniques and finally prove our results.

\section{Spectra, operators and examples} \label{sec:examples}
\subsection{Spectrum and pseudospectra}
Given a bounded linear operator $A$ on a Banach space $X$, we denote its \emph{spectrum} and \emph{pseudospectra} \cite{TrefEmb}, respectively, by
\[
\sigma(A)\ :=\ \{\lambda\in\CC: A-\lambda \text{ is not invertible}\}
\]
and
\begin{equation} \label{eq:speps}
\textstyle
\sigma_\eps(A)\ :=\ \{\lambda\in\CC: \|(A-\lambda)^{-1}\|>\frac 1\eps \},\qquad \eps>0,
\end{equation}
where we identify $\|B^{-1}\|:=\infty>\frac 1\eps$ if $B$ is not invertible, so that $\sigma(A)\subseteq \sigma_\eps(A)$ for all $\eps>0$.

\subsection{Band operators} 
We focus on operators on $\ell^2(\ZZ)$. 
Our basic building stones are \emph{multiplication operators}
$(M_ax)_n=a_nx_n$ with some $a\in\ell^\infty(\ZZ)$ and 
the \emph{shift operator}, $(S x)_n = x_{n-1}$, both acting on $\ell^2(\ZZ)$.

A \emph{band operator} is a finite sum of finite products of those two ingredients,
i.e.,
\begin{equation*}
A\ =\ \sum_{k=-w}^w M_{a^{(k)}}S^k
\end{equation*} 
with all $a^{(k)}\in\ell^\infty(\ZZ)$. The number $w\in\NN$ here is called the \emph{band-width} of $A$.

We identify an operator $A$ on $\ell^2(\ZZ)$ with its usual \emph{matrix representation} $(A_{ij})_{i,j \in \ZZ}$ with respect to the canonical basis in $\ell^2(\ZZ)$. Band operators exactly correspond to infinite band matrices with bounded diagonals.

Band operators act boundedly on every $\ell^p(\ZZ)$ with $p\in[1,\infty]$. 
Even their invertibility and spectrum are independent of $p$ (\cite[5.27]{Kurbatov.1999} and \cite[Cor.~2.5.4]{RaRoSi:Book}):
\begin{proposition}\label{prop:kurbatov}
    Let $B$ be a band operator on $\ell^p(\II)$ for some $\II \subseteq \ZZ$ and $p \in [1,\infty]$.
    If~$B$ is invertible, then $B$ is also invertible as an operator on $\ell^q(\II)$ for all $q \in [1,\infty]$.
\end{proposition}
So when we focus on $p=2$ for spectral analysis, that is not much of a restriction.
Many of our results will be formulated for the following class of band operators.

\subsection{Generalized Schrödinger operators}
Discretizing the 1D Schrödinger operator $\Delta+ v\cdot$ by finite differences, one gets a tridiagonal self-adjoint matrix with only one varying diagonal: the main diagonal. We generalize this as follows:

\begin{definition}[Generalized Schrödinger operator] \label{def:gSchrö}
For $b\in\ell^\infty(\ZZ)$, let $H(b)$ denote 
the \emph{generalized (discrete) Schrödinger operator},
\begin{equation} \label{eq:gSchroe} 
H(b)\ :=\ H(b,L,\gamma)\ :=\ L\ +\ S^\gamma M_b\,,
\end{equation}
on $\ell^2(\ZZ)$, where $L$ is a fixed translation invariant ($LS=SL$) band operator and $\gamma\in\ZZ$. The sequence $b$ is called \emph{the potential} of $H(b)$.
\end{definition}

The matrix of $L$ has constant diagonals, only finitely many of them nonzero. We add the sequence $b$ to the $\gamma$-th diagonal of $L$, resulting in $H(b)$. In particular, $H(b)$ is a band operator. Denote its bandwidth by $w$ throughout.
The standard and self-adjoint (discrete) Schrödinger operator is a special case of $H(b)$,
when $L=S+S^{-1}$, $\gamma=0$, $w=1$ and $b$ is real-valued.

\subsection{Classes of potentials considered here}\label{sec:classes}
Our examples have either pseudoergodic or aperiodic potentials. We explain these in the language of finite subwords.

\subsection*{Words} 
An \emph{alphabet} is a non-empty compact set $\S\subseteq\CC$. 
The elements of an alphabet are called letters.
For $n\in\NN$, vectors $w=(s_1,\dots,s_n)\in\S^n$ are referred to as \emph{words} over $\S$. We also write $w=s_1\dots s_n$ or $w:\{1,\dots,n\}\to\S$ with $w(k)=s_k$ and call $|w|:=n$ the \emph{length} of $w$.

Together with the empty word, $\epsilon$ with $|\mathcal \epsilon|=0$, $\S^0 = \{\epsilon\}$, and the operation of word concatenation, $(v_1,\dots,v_m)(w_1,\dots,w_n):=(v_1,\dots,v_m,w_1,\dots,w_n)\in\S^{m+n}$, the set $\S^* := \cup_{n=0}^\infty\ \S^n$ of all finite words forms a monoid -- the so-called \emph{free monoid over $\S$}.
We also study infinite words $w\in\S^\II$ with infinite $\II\subseteq\ZZ$, understood as $w:\II\to\S$.
For words $w\in\S^*$, $b\in\S^\II$ with $\II\in\{\ZZ,\NN\}$ and $N\in\NN$, we write
\\[-4mm]
\begin{itemize} \itemsep0mm
\item[$\circ\ $] $\mathrm{pos}(w,b):=\{k\in \II: b(k)\dots b(k+|w|-1)=w\}$,
\item[$\circ\ $] $\#(w,b):=|\mathrm{pos}(w,b)|$,\ \ meaning the number of occurrences,
\item[$\circ\ $] $\W(b) \ :=\ \{w\in\S^*:\#(w,b)\ge 1\}$, the set of all finite subwords of $b$, and
\item[$\circ\ $]  $\W_N(b):=\{w\in\W(b):|w|=N\}$.
\end{itemize}

\subsection*{Pseudoergodicity}
$b\in\S^\ZZ$ is called \emph{pseudoergodic} over the finite alphabet $\S$
if $\W_N(b)=\S^N$ for every $N\in\NN$.

This notion was introduced by Davies~\cite{Davies2001:PseudoErg} (and has been successfully employed
since then, e.g.~\cite{Davies2001:PseudoErg,LiBook,Chandler-WildeLindner.2011,Chandler-WildeLindner.2016,Colb:PE_pBC}) in order to capture spectral properties of random operators while eliminating stochastic details. Our results are very much in line with this philosophy of connecting spectral properties purely with subword
variety instead of locations, probability and distribution.

This spirit is reflected by the (anything but random) construction of our approximants: For $m\in\NN$, let $b_m$ be the periodic extension of a concatenation of all elements of $\S^m$. 
Listing those $|\S|^m$ words of length $m$ takes $m|\S|^m$ letters if done naively and $|\S|^m$ in a clever condensed arrangement as a de Bruijn sequence~\cite{deBruijn}; 
for example, the word $u=000\,001\,010\,011\,100\,101\,110\,111$ and
the cyclic word $v=00010111$ both contain all $w\in\{0,1\}^3$.
For both arrangements, Theorem~\ref{thm:approxspec} applies, 
so that $\sigma_\eps(H(b_m))\to \sigma_\eps(H(b))$,
in particular, condition~\eqref{eq:ev_equalsubwords} holds with $m_0:=N$.

\subsection*{Aperiodicity}
Let $\S=\{0,1\}$. Also \emph{aperiodicity} of $b\in\S^\NN$ is characterized by the element count of $\W_N(b)$: But, unlike for pseudoergodicity, where this count is $2^N$ (hence, maximal), for aperiodicity it is $N+1$ (which is minimal among not eventually periodic sequences, by the Morse-Hedlund theorem \cite{CovenHedlund}). 

A sequence $b\in\S^\ZZ$ is aperiodic if both its restrictions, to the negative and to the non-negative half-axis, are aperiodic. 
This is not equivalent to $|\W_N(b)|=N+1$; the latter is already satisfied by bi-infinite sequences as simple as $\chi_{\{0\}}$ or $\chi_\NN$, 
where $\chi_J$ is the characteristic function of a set $J\subseteq\ZZ$.

One of the most famous aperiodic words is the \emph{Fibonacci word} $b\in\S^\ZZ$ with
\begin{equation} \label{eq:Fibo}
b(n)\ =\ \chi_{[1-\alpha,1)}(n\alpha \bmod 1),\qquad n\in\ZZ,
\end{equation}
where $\alpha=\frac 12(\sqrt 5-1)$.

Writing down periodic approximants $b_m$ to an aperiodic $b$ such that \eqref{eq:ev_equalsubwords} holds with an explicitly given $m_0$ is not as simple as in the pseudoergodic case since putting ``legal'' subwords $u,v\in\W_N(b)$ one after the other often creates ``illegal'' ones, $w\not\in\W_N(b)$, in the transition zone from $u$ to $v$. 
For the Fibonacci word, $b_m$ can be constructed via~\eqref{eq:Fibo} with $\alpha$ replaced by the continued fraction expansion of $\alpha$, truncated after $m$ divisions~\cite{Gabel.2021}. Below is another approach for the Fibonacci word and some of its relatives.

\subsection*{Words generated by primitive substitutions}
Let us also come to a third class of potentials $b$ that has some overlap with the second, for example, the Fibonacci word \eqref{eq:Fibo}. 
Take a map $M:\Sigma \to \Sigma^* \setminus \{\epsilon\}$ and extend it via concatenation,
$M(uv):=M(u)M(v)$, to an endomorphism on $\Sigma^*$ and even on $\Sigma^\NN$.
We call such a morphism $M$ a \emph{substitution} if it satisfies 
\begin{enumerate}[label=(\roman*)]
 \item\label{it:subs11} There is an $a\in\Sigma$ with $M(a)=au$ for some $u\in\Sigma^*$,
 \item For all $c\in\Sigma$, $|M^n(c)| \to \infty$ for $n\to\infty$.
\end{enumerate}
Then $M^n(a)$ converges pointwise to the \emph{substitution word} $d \in \Sigma^\NN$, a fixed point of $M$.
If there is a $k\in\ZZ_+$ such that for all $c_1,c_2 \in\Sigma$, $c_2$ is a subword of $M^k(c_1)$, then we call $M$ \emph{primitive}.
For the Fibonacci word, the primitive substitution is given by 
\begin{equation}
M_{\text{Fib}}:\quad 0\mapsto 1, \qquad 1\mapsto 10.
\end{equation}

Other famous primitive substitutions on $\{0,1 \}^*$ include the \emph{Thue-Morse} substitution,
\begin{equation}
M_{\text{TM}}:\quad 0\mapsto 01, \qquad 1\mapsto 10,
\end{equation}
and the \emph{period doubling} substitution, 
\begin{equation}
M_{\text{PD}}:\quad 0\mapsto 01, \qquad 1\mapsto 00.
\end{equation}
 
To obtain two-sided infinite words $b\in\Sigma^\ZZ$ from a substitution word $d\in\Sigma^\NN$, we take accumulation points of $S^{-k}(d)$, $k\to\infty$, where $S$ is the right-shift.
These are exactly the words $b\in\Sigma^\ZZ$ that satisfy $\W(b) = \W(d)$.
For a given word $b\in \Sigma^\ZZ$ 
that is generated by a primitive substitution as described above, we can approximate spectral quantities of $H(b)$ by those of $H(b_m)$,  with $b_m=M^m(a)$ for $a\in\Sigma$ with~\ref{it:subs11}, 
see also~\cite{Beckus.2020}.

\subsection{Operator examples from mathematical physics} \label{sec:ex_math_phys}
\subsection*{Anderson model: self-adjoint and non-self-adjoint}
The famous Anderson model~\cite{Anderson58} of 1958 studies localization and delocalization of eigenvectors for a better understanding of electric conductivity in 1D disordered media.
In this model, one looks at the self-adjoint Schrödinger operator $H(b)$ from~\eqref{eq:gSchroe} on the axis with
$L=S+S^{-1}$, $\gamma=0$, $b\in\S^\ZZ$ pseudoergodic and $\S\subseteq \RR$. 
It is straightforward to prove (e.g.~\cite{LindnerRoch2010}) that $\sigma(H(b))=\S+[-2,2]$.
\medskip

In the late 1990s, the Anderson model reemerged in a non-self-adjoint (NSA) setting: 
The only change is that now $L=e^gS+e^{-g}S^{-1}$, where $g>0$ is the strength of an external magnetic field. 
The now also famous paper of Hatano and Nelson~\cite{HatanoNelson1997} looks at flux lines in type II superconductors under the influence of a tilted external magnetic field.
Within short time, the NSA Anderson model reappeared in population dynamics~\cite{NelsonShnerb} and other areas.

Mathematicians from stochastics~\cite{GoldKoru} and spectral theory~\cite{Davies2001:SpecNSA,Davies2001:PseudoErg,MartinezThesis} 
studied its spectrum and how it invades the complex plane, and the subject of pseudospectra (see~\cite{TrefEmb} and the references therein) received an additional uplift.

\nofigure{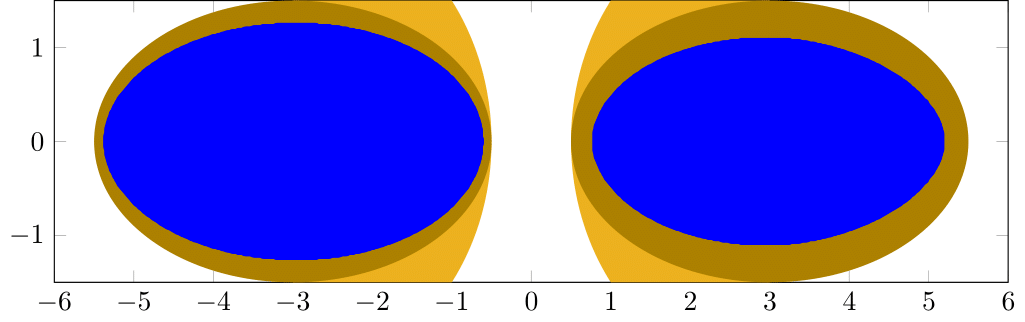}{
Here is an approximation of the pseudospectrum of the NSA Anderson model with $e^g=\frac 12$, i.e.~$H(b)$ with  $L=\frac 12S+ 2S^{-1}$, $\gamma=0$,  
$\S=\{-3,3\}$ and $b\in\S^\ZZ$ pseudoergodic. We see in blue the pseudospectrum of $H(b_{14})$~\cite{Gabel.2023}, 
where $b_m$ is a periodic sequence over $\S$ containing all subwords of $b$ of length $m$. By Theorem~\ref{thm:approxspec}, 
$\sigma_\eps(H(b_m))$ approximates $\sigma_\eps(H(b))$ in Hausdorff distance as $m\to\infty$. For the spectrum, 
known bounds, e.g.~from Theorem 14 in \cite{Davies2001:PseudoErg}, are displayed: dark orange is guaranteed to be spectrum, 
bright orange shows how far the spectrum could go at most.
}
\label{fig:NSAAnderson}

\subsection*{Quasicrystals and the Fibonacci Hamiltonian: self-adjoint and non-self-adjoint}
Again, the question was electric conductivity but now the medium was not disordered (random) but fairly ordered (periodic) -- or maybe not quite -- when Shechtman~\cite{Shechtman.1984} discovered the first so-called quasicrystal in 1982 in his laboratory.

Mathematicians were excited \cite{Suto.1987,Bellissard.1989,Stollmann,Damanik.2015,Damanik.2016} to see Hamiltonians with zero measure Cantor spectrum that was purely singular continuous, and finally, in 2011, after years of being called a quasiscientist, Shechtman received the Nobel Prize.
The phrase ``aperiodic'' was coined to label this scenario of slightly disorderly order.

The most famous model in 2D is the Penrose tiling, whose spectral analysis is so notoriously difficult that analysts resort to 1D and, e.g., the Fibonacci Hamiltonian~\cite{Damanik.2015}. 
In Figures \ref{fig:FibHam} and~\ref{fig:NSAFibHam}, we approximate both, the standard (self-adjoint) and a modified (non-self-adjoint) Fibonacci Hamiltonian $H(b)$, by periodic potentials $b_m$ that have $\W_N(b_m)=\W_N(b)$ for $N=N(m)\to\infty$ as $m\to\infty$. By Theorem \ref{thm:approxspec}, $\sigma_\eps(H(b_m))$ approximates $\sigma_\eps(H(b))$ in Hausdorff distance as $m\to\infty$. In the self-adjoint case, the same holds for spectra.

\nofigure{tikz//FibHam/FibHam_standalone}{
Here is an approximation of the spectrum of the standard (self-adjoint) Fibonacci Hamiltonian, $H(b)$ with  $L=S+S^{-1}$, $\gamma=0$ and  $b$ from~\eqref{eq:Fibo}. 
The spectra of the periodic approximations (here shown in blue), $H(b_m)$, are a union of finitely many closed intervals and Hausdorff-approximate, by Theorem~\ref{thm:approxspec} (the normal case), 
the spectrum of $H(b)$, which is known to be a Cantor set on the real line~\cite{Damanik.2016}.
Here we show, stacked in the vertical direction, many copies of the real line together with an $m$ axis to better envisage this approximation, $\sigma(H(b_m))\to \sigma(H(b))$ as $m\to\infty$.
}
\label{fig:FibHam}

\nofigure{tikz/FibHamNSA/FibHamNSA_scaled}{
In contrast to Figure~\ref{fig:FibHam}, here is an approximation of a NSA Fibonacci Hamiltonian. 
The potential $b$ from~\eqref{eq:Fibo} was multiplied by minus the imaginary unit. 
We see the spectrum (dark blue) and the $\varepsilon$-pseudospectrum (light blue, $\varepsilon=10^{-4}$) of $H(b_{18})$~\cite{Gabel.2023}.
}
\label{fig:NSAFibHam}
\pagebreak

\subsection*{Feinberg \& Zee's randomly hopping particle}
Motivated by 
\cite{HatanoNelson1997},
other NSA models popped up in the late 1990s, for example, 
Feinberg \& Zee's paper~\cite{FeinZee97} on localization and delocalization in models of a randomly hopping particle on a 1D grid. 
The particle can jump one node to the left or right, and it has a state (e.g.~spin) in $\S=\{\pm 1\}$ that changes randomly at every jump. 
Also the model was studied where the state changes randomly when jumping right but not when jumping left.
By a similarity transform, the two particles can be shown to have identic spectra.

\nofigure{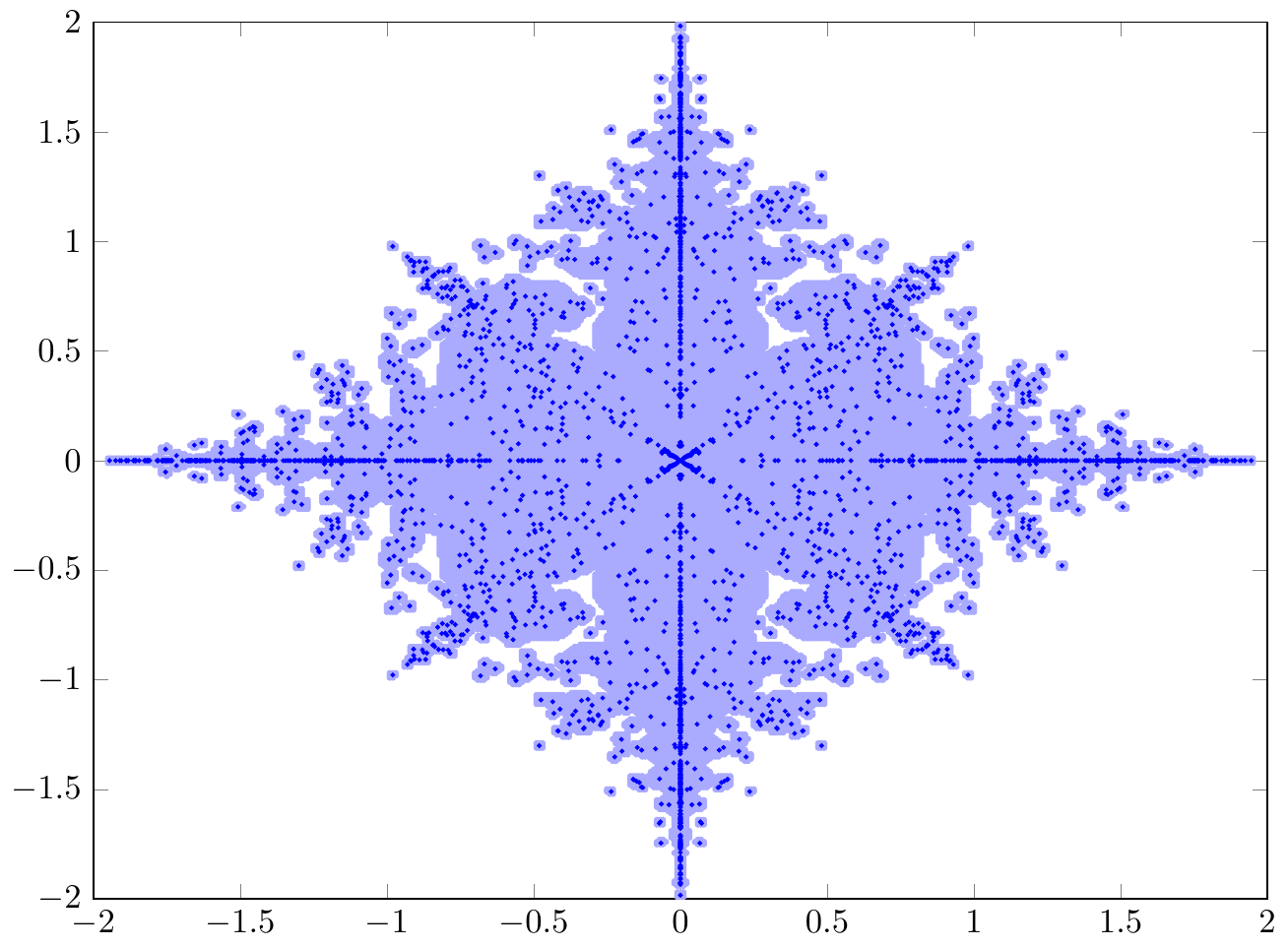}{
Approximation of $\sigma(H(b))$ for the hopping particle with $L=S$, $\gamma=-1$, $\S=\{\pm 1\}$ and $b\in\S^\ZZ$ pseudoergodic. 
We see the spectrum and pseudospectrum of $H(b_{12})$~\cite{Gabel.2023}.
}\label{fig:hop2}

The spectrum looks self-similar; it is invariant under many transformations \cite{Hagger:symm} but it is not explicitly known. 
See~\cite{HolzOrlZee,CWChML:3diag,Hagger:dense,Hagger:symm,CW.Hagger} for extensive studies, including provable subsets and supersets. 
Figure~\ref{fig:hop2} shows a spectral approximation of $H(b)$ by $H(b_m)$ with our periodic approximation $b_m$ of a pseudoergodic $b$.

\nonofigure{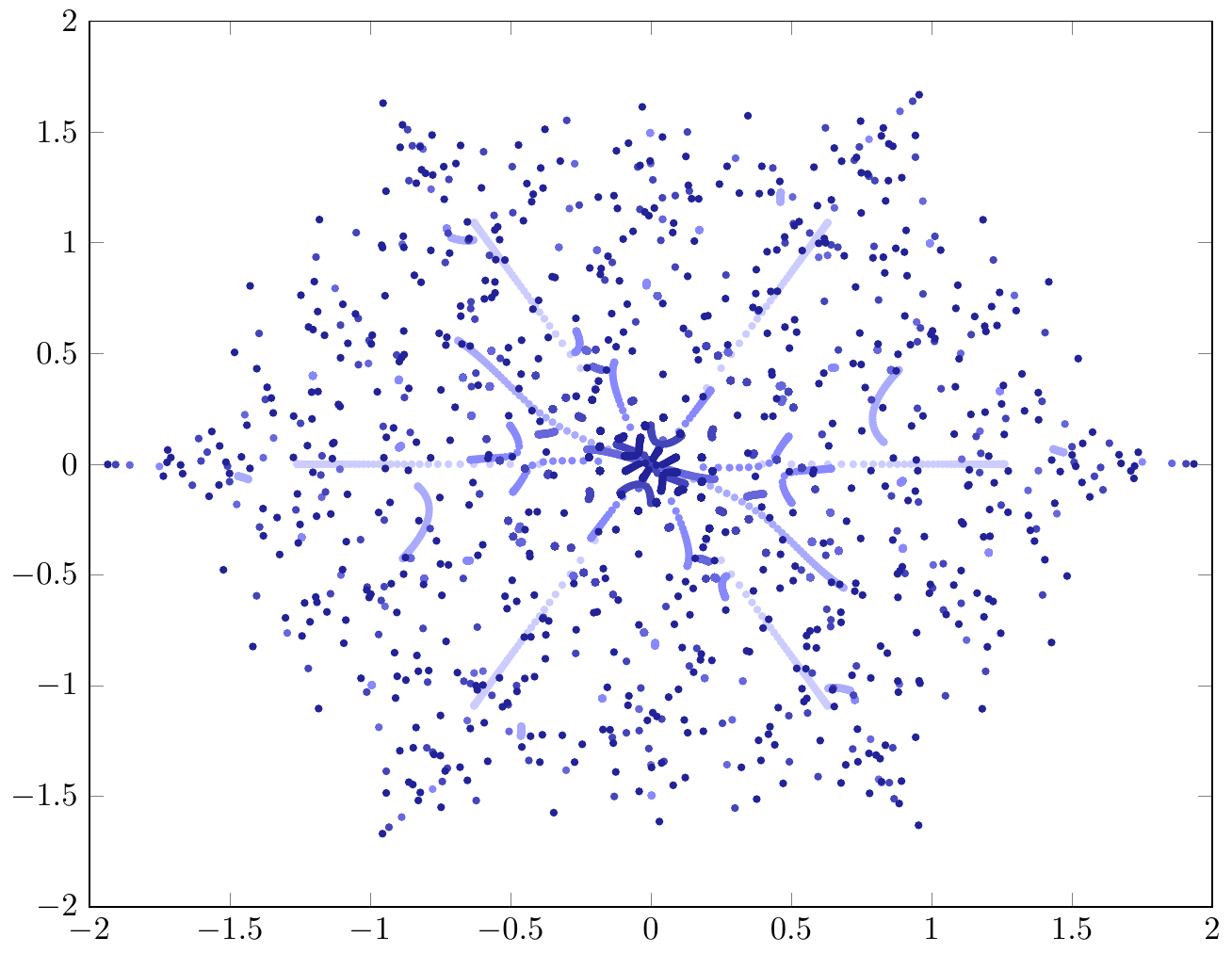}{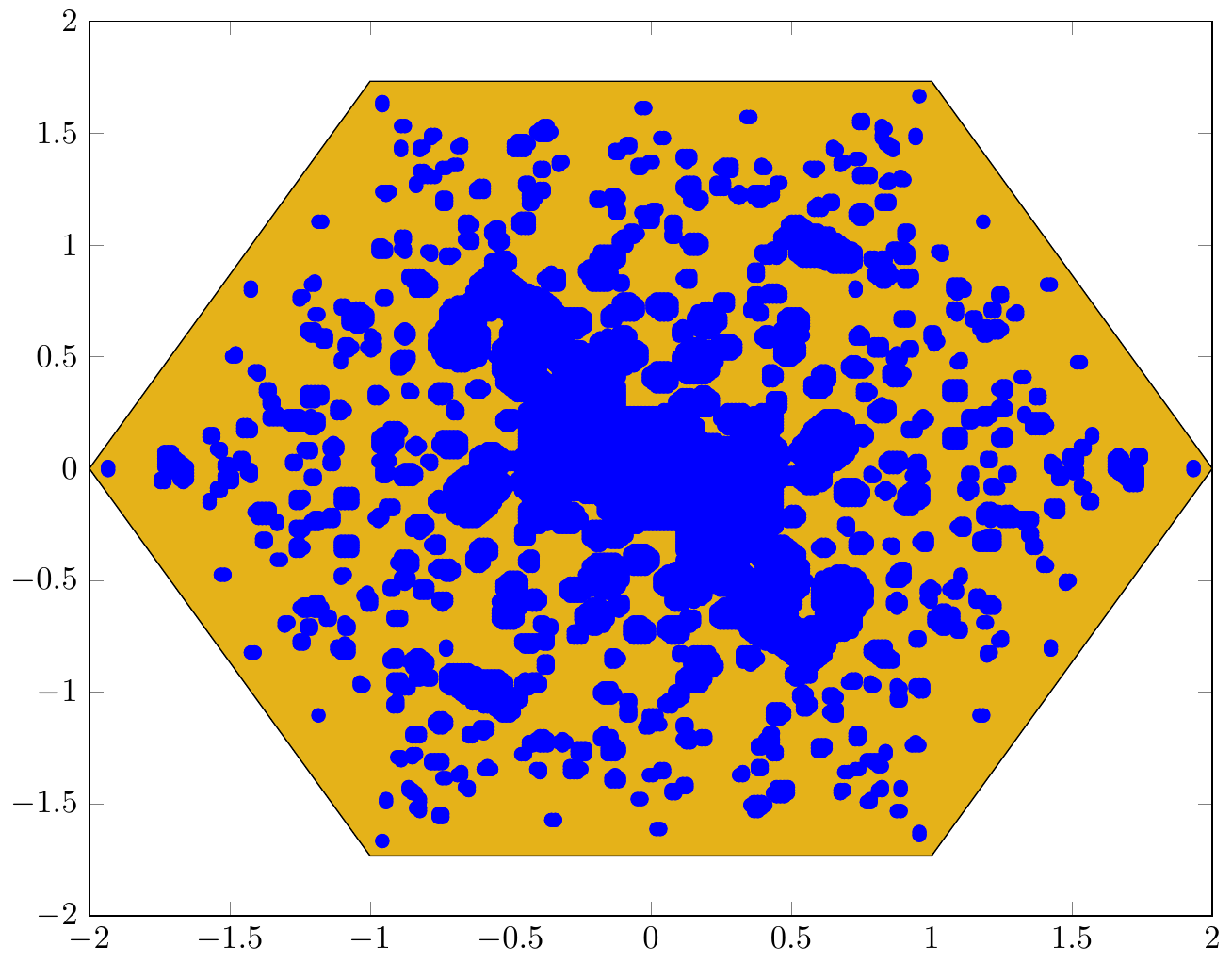}{
Approximation of $\sigma(H(b))$ for the hopping particle with $q=3$ states, with $L=S$, $\gamma=-1$, $\S=\{z\in\CC:z^3=1\}$ and $b\in\S^\ZZ$ pseudoergodic. 
On the left we see the union of spectra of $H(b_m)$ for $m=1,\ldots,6$~\cite{Gabel.2023}, each $m$ corresponds to a different shade of blue.
On the right the pseudospectrum of $H(b_{6})$  in blue and the numerical range of $H(b)$ in yellow.}\label{fig:hop3}

Cicuta, Contedini and Molinari \cite{CicutaContediniMolinari2000} generalized \cite{FeinZee97} to a particle with $q\in\NN$ different states, say $\S=\{z\in\CC:z^q=1\}$.
Many results of the case $q=2$ are preserved, see \cite{CWDavies2011,Weber.2022}. 
The numerical range (an upper bound on the spectrum) is a regular $2q$-polygon with outer radius 2 and center at the origin, see Figure \ref{fig:hop3}.

\subsection*{A one-way model}
Brezin, Feinberg and Zee \cite{BrezinZee,FeinZee97} also modeled a random particle that can only jump in one direction and then randomly change its state, a so-called one-way model.
The spectral analysis is simpler as the matrix is only supported on two adjacent diagonals. That's why the spectrum is explicitly known, 
see~\cite{TrefContEmb} for one constant and one random diagonal and~\cite{LiBiDiag} for two stochastically independent random diagonals.
We display a spectral approximation of the latter case in Figure~\ref{fig:bidiag}.

\nofigure{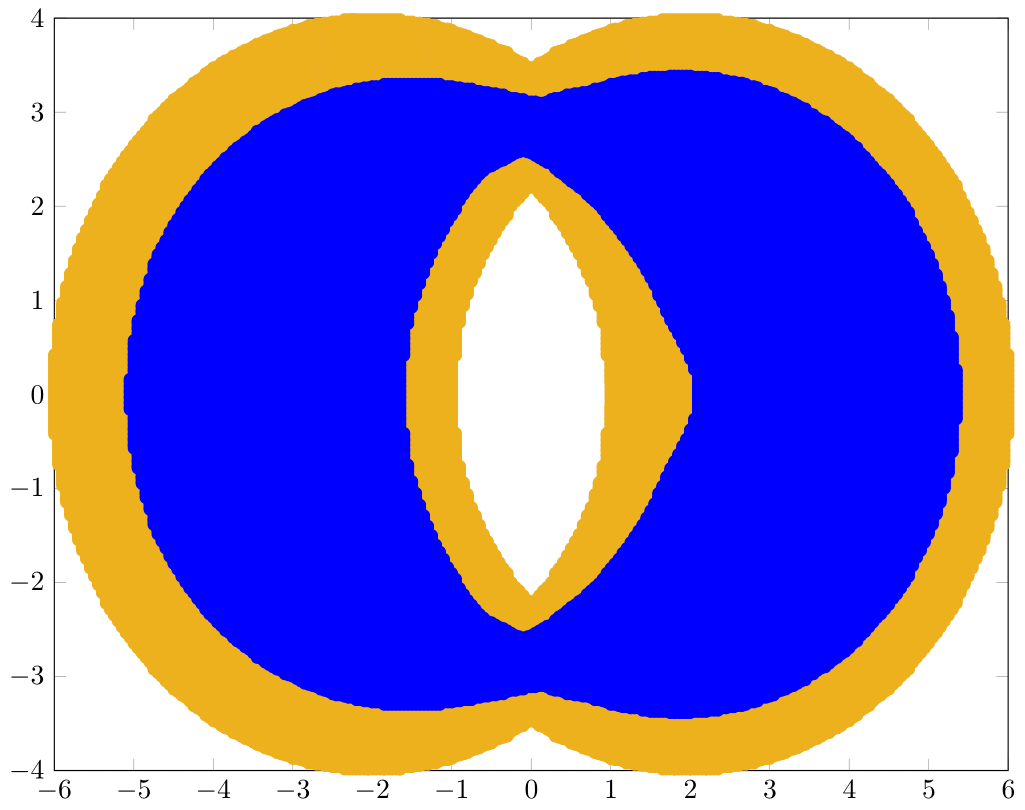}{
Here is an approximation of the one-way model $H(b,c):=M_b+SM_c$ with two pseudoergodic diagonals $b\in\S_b^\ZZ$ and $c\in\S_c^\ZZ$. 
Here we use $\S_b=\{-2,2\}$ and $\S_c=\{3,4\}$. 
This case is still subject to Proposition~\ref{prop:approxspec}.
We see in blue the pseudospectrum of $H(b_7,c_7)$~\cite{Gabel.2023}, where $b_m$ and $c_m$ are periodic 
and the pair $(b_m,c_m)$ contains all $\S_b\times\S_c$-subwords of $(b,c)$ of length $m$. 
In orange, we see $\sigma(H(b,c))$ based on~\cite{LiBiDiag}.
}
\label{fig:bidiag}

\section{Further notations and tools} \label{sec:notations}

\subsection*{Discrete intervals and submatrices}
We use the following abbreviations for discrete intervals. Given $a,b\in\ZZ$, we write
\begin{eqnarray*}\itemsep0mm
a..b&\ :=\ &\{n\in\ZZ:a\le n\le b\},\\
a..&\ :=\ &\{n\in\ZZ:a\le n\},\\
..b&\ :=\ &\{n\in\ZZ:n\le b\}.
\end{eqnarray*}

For an operator $A:\ell^2(\ZZ)\to\ell^2(\ZZ)$ with matrix representation $(A_{ij})_{i,j\in\ZZ}$
and a discrete interval $\II\subseteq\ZZ$, we abbreviate the restriction of $A$ to $\ell^2(\II)$,
\[
A|_{\ell^2(\II)}\ :\ \ell^2(\II)\to\ell^2(\ZZ),
\]
by $A|_\II$. 
The corresponding matrix representation is $(A_{ij})_{i \in\ZZ,j\in\II}$.
Together with our shorthands for discrete intervals, this explains the notations
$A|_{a..b},\ A|_{a..}$ and $A|_{..b}$.

Furthermore, the operator $\ell^2(\NN)\to\ell^2(\NN)$, corresponding to $A^+:= (A_{ij})_{i,j\in\NN}$, is called the \emph{compression of $A$ to $\NN$}.

\subsection*{Approximate equality and bounds}
We, moreover, find the following notations useful: 
For $a,b\in\mathbb R$ and $\eps>0$, let us write $a\stackrel\eps\approx b\ $ if $\ b\in (a-\eps,a+\eps)$ 
and $a\stackrel\eps\preceq b\ $ if $\ b\in [a,a+\eps)$.
In particular, if $a \stackrel\delta\approx b$ and $b \stackrel\eps\approx c$ then $a \stackrel{\delta+\eps}\approx c$.
The same holds for the relation $\stackrel\eps\preceq$.

\subsection*{The lower norm}
An important spectral quantity that we use a lot in our arguments is the so-called \emph{lower norm} of an operator $A$ on $\ell^2(\II)$, meaning
\begin{equation}
\nu(A)\ :=\ \inf \{\|Ax\|: x\in \ell^2(\II), \|x\|=1\}.
\end{equation}
Note that it is not a norm and that the name ``lower norm'' is used, as in \cite{LiBook, RaRoSi:Book}, to address its role as a counterpart to the operator norm, $\|A\|\coloneqq \sup_{\|x\|=1}\|Ax\|$. 

$\nu(A)$ turns out to be a fairly accessible quantity to study $\|A^{-1}\|$. Indeed,
\begin{equation} \label{eq:invnu}
\|A^{-1}\|\ =\ 1/\min\big\{\,\nu(A),\nu(A^*)\,\big\}.
\end{equation}
Here $A^*$ is the adjoint of $A$, and $\nu(A)=\infty$ if and only if $A$ is not invertible. 
In Hilbert space, 
$\nu(A)$ is the smallest singular value of $A$. For normal $A$, it is the smallest (in modulus) spectral value, 
\begin{equation} \label{eq:nu_distspec}
\nu(A)=\dist(0,\sigma(A)),
\quad\text{whence}\quad
\nu(A-\lambda)=\dist(\lambda,\sigma(A)).
\end{equation}
For non-normal operators, we have ``$\le$'' instead in both equalities of~\eqref{eq:nu_distspec}.

For band operators, $\nu(A)$ can be conveniently approximated / localized via 
\begin{equation} \label{eq:localnu}
  \nu_N(A)\searrow \nu(A)\quad\text{as}\quad N\to\infty,
\end{equation}
where $\nu_N(A)$, for $N\in\NN$, refers to the \emph{local lower norm}, defined as
\begin{equation} \label{eq:nu_n}
\nu_N(A)\ \coloneqq\ \inf\big\{\|Ax\|: x\in\ell^2(\II), \|x\|=1, \diam(\supp(x))<N\big\},
\end{equation}
where $\supp(x) = \{k \in \II : x_k \neq 0\}$ and $\diam(S)=\sup_{s,t\in S}|s-t|$. 
Together with \eqref{eq:invnu} we get, as $N\to\infty$,
\begin{equation} \label{eq:nu_n2}
\min\big\{\nu_N(A),\nu_N(A^*)\big\}\ \searrow\ \min\big\{\nu(A),\nu(A^*)\big\}\ =\ 1\,/\,\|A^{-1}\|.
\end{equation}
In Lemma \ref{lem:self2} below we will see for which kinds of operators $A$ we can
neglect the terms involving $A^*$ in \eqref{eq:invnu} and \eqref{eq:nu_n2}.
For 
the approximation of spectral quantities, it is important to know that \eqref{eq:localnu} holds in a very uniform sense, as follows:
\begin{lemma}[{\cite[Prop. 6]{BigQuest}}] \label{lem:nuN}
Let $\eps >0$, $r>0$ and $w\in\NN$.
Then there is an $N\in\NN$ such that, for all band operators $A$ with band-width less than $w$ and $\|A\| < r$, 
\begin{equation*}
\nu(A) \stackrel\eps\preceq \nu_N(A)\,.
\end{equation*}
\end{lemma}
One can explicitly quantify $N$ vs.~$\eps$. 
An analogous localization of the operator norm is in~\cite[Prop.~3.4]{HagLiSei}.

For an even finer localization of the lower norm, put, for $\II\subseteq\ZZ$,
\begin{equation} \label{eq:nu-lr}
\nu_\II(A)\ :=\ \inf\{\|Ax\|:\supp(x)\subseteq \II, \|x\|=1\}.
\end{equation}
\begin{corollary} \label{cor:nuN}
For every band operator $A$ and all $\eps>0$, there are $l,r\in\ZZ$ such that
$\nu(A)\stackrel\eps\preceq\nu_{l..r}(A)$.
\end{corollary}
\begin{proof}
With the help of Lemma~\ref{lem:nuN} choose $N$ large enough that $\nu(A)\stackrel{\eps/2}\preceq \nu_N(A)=\inf_j\nu_{j..j+N-1}(A)$
and then $j$ such that the infimum is $\stackrel{\eps/2}\preceq \nu_{j..j+N-1}(A)$. 
Then put $l:=j$ and $r:=j+N-1$.
\end{proof}

\subsection*{Submatrices of consecutive columns}
Let $A$ be a band matrix with band-width $w$.

For $N\in\NN$, we say that $C$ is an \emph{$N$-column submatrix} of $A$ if $C$ consists of $N$ consecutive columns of $A$. 
For simplicity and comparability, restrict $C$ to size $(N+2w)\times N$, capturing exactly the ``banded part'' of those $N$ columns of $A$, 
that is,
$C=(C_{ij})_{i\in 1-w..N+w,\ j\in 1..N}$ with 
\[
\exists k\in\ZZ:\quad C_{i,j}=A_{k+i,k+j},\quad i\in 1-w..N+w,\ j\in 1..N,
\]
where $C_{i,j}:=0$ if $A_{k+i,k+j}$ is not defined.
Let
\begin{itemize}
\item $\C_N(A)$ denote the set of all $N$-column submatrices of $A$ and 
\item $\C(A):=\cup_{N\in\NN}\ \C_N(A)$.
\end{itemize}

\subsection*{Self-contained operators}
We call a band operator $A$ \emph{self-contained} if every $C\in\C(A)$ appears infinitely often in $A$.
For $A=H(b)$ with some word $b$ this means that every $w\in\W(b)$ appears infinitely often in $b$.
For example, operators with aperiodic or pseudoergodic matrix diagonals (see Section \ref{sec:classes}) are self-contained.
\begin{lemma}\label{lem:self1}
For a band operator $A$, the statements
\[
\begin{array}{rlp{5mm}}
(i)&A\text{ is self-contained,}&\\
(ii)&A \text{ is self-similar (in the sense of \cite{CWLi:Favard,Li:MinLimOps}),}&\\
(iii)&A \text{ is invertible if and only if it is Fredholm,}
\end{array}
\]
are related as follows: $\quad (i) \Rightarrow (ii) \Rightarrow (iii)$.\\
If the set of all entries of $A$ (as a matrix) is finite then $(i)\Leftrightarrow(ii)$.
\end{lemma}
\begin{proof}
This is immediate from Lemmas 4.3 and 2.2 in \cite{Li:MinLimOps}.
\end{proof}
For self-contained operators, many of our formulas and arguments simplify:
\begin{lemma} \label{lem:self2}
If $A$ is self-contained, then $\nu(A)=\nu(A^*)$, 
so that we can discard $\nu(A^*)$ and $\nu_N(A^*)$ from \eqref{eq:invnu} and \eqref{eq:nu_n2}, i.e.
\begin{equation*}
\nu_N(A) \searrow \nu(A) =\frac 1{\|A^{-1}\|} \quad \text{as } N\to \infty.
\end{equation*}
\end{lemma}
\begin{proof}
We distinguish two cases for $\nu(A)$.
\begin{itemize} 
\item Case 1: $\nu(A)>0$
\begin{itemize}
\item then $\dim \ker(A)=0$ and $\im(A)$ is closed, e.g.~Lemma 2.32 in \cite{LiBook},
\item so $A$ is semi-Fredholm (in terms of~\cite{Sei:semi}, $\Phi_+$),
\item by Theorem 4.3 in \cite{Sei:semi}, $A$ is Fredholm, 
\item by Lemma \ref{lem:self1} $(iii)$, $A$ is invertible,
\item but then $\nu(A^*)=\nu(A)>0$, e.g.~Lemma 2.10 \cite{HagLiSei}. 
\end{itemize}
\item Case 2: $\nu(A)=0$\\
If $\nu(A^*)$ were nonzero then, arguing as in case 1, also $\nu(A)>0$; contradiction.
So $\nu(A^*)=0$. \qedhere
\end{itemize}
\end{proof}

\section{Subwords, spectra and approximation} \label{sec:results}
\subsection{Column-submatrices, spectra and pseudospectra}
Because the submatrices $C\in\C_N(A)$ are what vectors $x$ with $\diam(\supp(x))<N$ get to ``see'' of $A$, the set $\C_N(A)$ has obvious connections to $\nu_N(A)$ and $\nu(A)$ but then also to resolvent norm and spectrum of $A$:
\begin{lemma} \label{lem:nuN-via-CN}
Let $A$ be a band operator on $\ell^2(\II)$ with a discrete interval $\II\subseteq\ZZ$. Then, for every $N\in\NN$,
\[
\nu_N(A)\ =\ \inf  \{\nu_{j..j+N-1}(A): j..j+N-1\subseteq \II\} =\ \inf\{\nu(C):C\in\C_N(A)\}.
\]
\end{lemma}
\begin{proof}
Rewriting the definition \eqref{eq:nu_n} in terms of \eqref{eq:nu-lr}, gives the first equality.
The second equality is by the definition of $N$-column submatrices and $\C_N(A)$.
\end{proof}

\begin{proposition} \label{prop:NCC}
Let $\II_A,\JJ_A,\II_B,\JJ_B\subseteq\ZZ$ be finite or infinite discrete intervals and let
$A:\ell^2(\JJ_A)\to\ell^2(\II_A)$ and $B:\ell^2(\JJ_B)\to\ell^2(\II_B)$ be band operators, associated with matrices $(A_{ij})_{i\in\II_A,j\in\JJ_A}$ and $(B_{ij})_{i\in\II_B,j\in\JJ_B}$ with the same band-width $w$.
\begin{enumerate}[label=(\alph*)]
\item If, for some $N\in\NN$, $\C_N(A)\subseteq \C_N(B)$ then $\nu_N(A)\ge\nu_N(B)$.
\item If $\C(A)\subseteq \C(B)$ then $\nu(A)\ge\nu(B)$.
\item If $N\in\NN$ and $\C_{N+2w}(A)\subseteq \C_{N+2w}(B)$ then also $\C_N(A^*)\subseteq \C_N(B^*)$.
\item If $\C(A)\subseteq \C(B)$ then also $\C(A^*)\subseteq \C(B^*)$.
\end{enumerate}
Now suppose $\II_A=\JJ_A$ and $\II_B=\JJ_B$, so that $A$ and $B$ are endomorphisms of $\ell^2(\II_A)$ and $\ell^2(\II_B)$, respectively.
\begin{enumerate}[label=(\alph*)]
\setcounter{enumi}{4}
\item If $\C(A)\subseteq \C(B)$ then
\begin{equation} \label{eq:nu-lambda}
\nu(A-\lambda)\ \ge\ \nu(B-\lambda),\quad \forall \lambda\in\CC
\end{equation}
and
\begin{equation} \label{eq:resAB}
\|(A-\lambda)^{-1}\|\ \le\ \|(B-\lambda)^{-1}\|,\quad \forall \lambda\in\CC,
\end{equation}
whence
\[
\sigma(A)\subseteq \sigma(B)
\qquad\text{and}\qquad
\sigma_\eps(A)\subseteq \sigma_\eps(B),\quad \eps>0.
\]
\end{enumerate}
\end{proposition}
\begin{proof}
\begin{enumerate}[label=(\alph*)]
\item 
This is immediate from Lemma \ref{lem:nuN-via-CN}.
\item If $\C(A)\subseteq\C(B)$, i.e.~$\C_N(A)\subseteq\C_N(B)$ holds for all $N\in\NN$, then,
by (a) and \eqref{eq:localnu}, it follows that $\nu(A) \ge \nu(B)$. 
\item 

Let $\C_{N+2w}(A)\subseteq\C_{N+2w}(B)$ and take $C\in\C_N(A^*)$.
Then $C$ has $N$ columns and $N+2w$ rows.
So $C^*$ is contained in $N+2w$ consecutive columns of $A$ and, hence,
in a matrix $D\in\C_{N+2w}(A)\subseteq\C_{N+2w}(B)$.
Using the same arguments backwards, $C\in\C_N(B^*)$.
\item If $\C(A)\subseteq\C(B)$ then $\C_{N+2w}(A)\subseteq\C_{N+2w}(B)$, and, by (c), 
$\C_N(A^*)\subseteq\C_N(B^*)$ for all $N\in\NN$, so that $\C(A^*)\subseteq\C(B^*)$.
\item From $\C(A)\subseteq\C(B)$ it follows that $\C(A-\lambda)\subseteq\C(B-\lambda)$ for all $\lambda\in\CC$, so that \eqref{eq:nu-lambda} follows in analogy.
By (d), we also have $\C((A-\lambda)^*)\subseteq\C((B-\lambda)^*)$ for all $\lambda\in\CC$,
so that \eqref{eq:nu-lambda} also holds for the adjoints, i.e.,
\[
\nu(A-\lambda)\ \ge\ \nu(B-\lambda)
\quad\text{and}\quad
\nu((A-\lambda)^*)\ \ge\ \nu((B-\lambda)^*)
\]
for all $\lambda\in\CC$. Now \eqref{eq:resAB} follows from \eqref{eq:invnu}.
The inclusion of spectra and pseudospectra is now immediate, by their definition.
\qedhere
\end{enumerate}
\end{proof}

\subsection{Two Schrödinger operators on the axis: subword variety and spectrum}
In the setting of a generalized Schrödinger operator, $A=H(b)$ with $b\in\S^\ZZ$,
the matrix is constant on all but one diagonal, so that $\C_N(H(b))$ corresponds directly to the set $\W_N(b)$ of all length-$N$ subwords of $b$. 
We get a first simple result on how the subword variety of $b$ has implications on resolvent and spectrum of $H(b)$.
\begin{theorem} \label{thm:subwordspec}
If $b,c\in\S^\ZZ$ with $\W(b)\subseteq\W(c)$ then $\nu(H(b))\ge\nu(H(c))$ and
\[
\|(H(b)-\lambda)^{-1}\|\le\|(H(c)-\lambda)^{-1}\|,
\quad \lambda\in\CC,
\]
so that
\[
\sigma(H(b))\subseteq\sigma(H(c))
\qquad\text{and}\qquad
\sigma_\eps(H(b))\subseteq\sigma_\eps(H(c)),\quad\eps>0.
\]
\end{theorem}
\begin{proof}
The result follows directly from Proposition \ref{prop:NCC} and \eqref{eq:invnu} since
$\W(b)\subseteq\W(c)$ implies $\C(H(b))\subseteq\C(H(c))$ and $\C(H(b)^*)\subseteq\C(H(c)^*)$.
\end{proof}

\subsection{Band operator: axis versus half-axis} 
Let $A$ be a band operator on the axis and let $w$ denote its band-width.
We want to compare the inverses, resolvent norms and spectra of $A$ and 
its half-axis compression $A^+$.

As an intermediate operator between $A$ and $A^+$, look at its restriction
\[
A|_\NN\ :=\ A|_{\ell^2(\NN)}\ :\ \ell^2(\NN)\to\ell^2(\ZZ)
\]
to $\ell^2(\NN)$.
The matrix of $A|_\NN$ is $\ZZ\times\NN$ and consists of the columns of $A$ with index in $\NN$;
the matrix of $A^+$ is $\NN\times\NN$ and consists of the rows of $A|_\NN$ with index in $\NN$:
\[
A\ =\ 
\left(
\begin{array}{ccc|ccccc}
\smash\ddots&\smash\ddots&\smash\ddots&\\
\smash\ddots&+&*&*\\
\smash\ddots&*&+&*&*\\ \hline
&*&*&+&*&*\\
&&*&*&+&*&*\\
&&&*&*&+&*&\smash\ddots\\
&&&&*&*&+&\smash\ddots\\
&&&&&\smash\ddots&\smash\ddots&\smash\ddots
\end{array}
\right)
\ \ 
\begin{tabular}{l}
The band-width here is $w=2$,\\
``$+$'' marks the main diagonal,\\[4mm]
the right half is $A|_\NN$,\\
the lower right quarter is $A^+$.
\end{tabular}
\]
Another way to look at these restrictions is that
\begin{equation} \label{eq:PAP}
A|_\NN=AP:\im(P)\to\ell^2(\ZZ)
\quad\text{and}\quad
A^+=PAP:\im(P)\to\im(P),
\end{equation}
where $P$ is the orthogonal projection from $\ell^2(\ZZ)$ onto $\im(P)=\ell^2(\NN)$.

Let us start with the assumption that 
\begin{equation} \label{eq:CA=CA'}
\C(A)\ =\ \C(A|_\NN),
\end{equation}
so that no other patterns appear on the columns with index $j\in ..0$ of $A$ and the only aspect here is what effect the truncation (or zero Dirichlet condition), from $A$ to $A^+$, has on the resolvent and the spectrum. 
Before we come to this, let us study some further implications of \eqref{eq:CA=CA'}.
\begin{lemma}\label{lem:auto_selfcontained}
Let $A$ be a band operator. 
If $A$ satisfies \eqref{eq:CA=CA'} then
\begin{enumerate}[label=(\alph*)]
 \item  $A$ is self-contained,
 \item  also $A^*$ is subject to \eqref{eq:CA=CA'} in place of $A$, and 
 \item one also has $\C(A)=\C(A|_{k..})$ for all $k\in\ZZ$. 
 \end{enumerate}
Self-containedness of a band operator $B$ can even be characterized via the restrictions  $B|_\NN$ and $B|_{-\NN}$ in the following sense:
\begin{enumerate}[label=(\alph*)]
\addtocounter{enumi}{3}
 \item  $B$ is self-contained if and only if $\C(B) = \C(B|_\NN) \cup \C(B|_{-\NN})$.
\end{enumerate}
\end{lemma}
\begin{proof}
\begin{enumerate}[label=(\alph*)]
 \item   
Let $C\in\C(A|_\NN)$ and let $l..r$ be the corresponding column numbers. 
We show that $C$ can be found in infinitely many positions of $A$:

By \eqref{eq:CA=CA'}, the submatrix $D\in\C(A)$ at columns $-r..r$ of $A$ can be found in $A|_\NN$, and hence at columns $1..2r+1$ or later. 
In particular, the rightmost $r-l+1$ columns of $D$, forming $C$, are found at columns $l+r+1..2r+1$ or later, which is disjoint from the location, $l..r$, (it is further to the right) of the original $C\in\C(A|_\NN)$.
Now keep repeating the argument for the newly found copy of $C$ in $A|_\NN$ to find a further copy of $C$, even further to the right. Hence, $C$ appears infinitely many times in $A$, i.e.~$A$ is self-contained.
\item Follows from Proposition~\ref{prop:NCC}~(c).
\item Follows from (a).
\item Since self-containedness of $B$ implies that every $C\in \C(B)$ appears infinitely often, it is clear that $C$ also appears in $B|_\NN$ or in $B|_{-\NN}$.
Hence it remains to show that  $\C(B) = \C(B|_\NN) \cup \C(B|_{-\NN})$ implies that $B$ is self-contained. But this can be done via similar arguments as in (a).
\qedhere
\end{enumerate}
\end{proof}

A first quick judgement on the spectrum of $A^+$ vs. that of $A$: The step from $A$ to $A|_\NN$ does not change the lower norm, by the assumption~\eqref{eq:CA=CA'} and Proposition~\ref{prop:NCC}~(b). But the step from $A|_\NN$ to $A^+$, chopping off some nonzero rows and hence deleting the entries $y_{1-w},\dots,y_0$ of every $y:=A|_\NN x$, decreases the lower norm and hence increases the inverse, resolvent and spectrum. 

The proper analysis is straightforward: 
For $x\in\ell^2(\NN)$, in the notations of \eqref{eq:PAP},
\[
	\|A^+x\|=\|PAPx\|\le\|APx\|=\|A|_\NN x\|=\|A\hat x\|, 
\]
where $\hat x\in\ell^2(\ZZ)$ is $x$, extended by zeros. 
But $PAPx=APx$ if $\supp(x)\subseteq w+1..\,$.
So, for all $N\in\NN$,
\begin{align}
\nu_{j..j+N-1}(A^+)\ &\le\  \nu_{j..j+N-1}(A|_\NN)\ =\  \nu_{j..j+N-1}(A),\quad j\in 1..w, \label{eq:columns_cut}\\
\nu_{j..j+N-1}(A^+)\ &=\   \nu_{j..j+N-1}(A|_\NN)\ =\  \nu_{j..j+N-1}(A),\quad j\in w+1..\ . \label{eq:columns_intact}
\end{align}
Columns with index in $1..w$ (probably) lost some nonzero entries when passing from $A$ via $A|_\NN$ to $A^+$; columns in $w+1..$ did not.
With this preparation, we prove:
\begin{proposition} \label{prop:band+}
Let $A$ be a band operator with band-width $w$ on the axis.
Then, for all $T\in\{A-\lambda, (A-\lambda)^*:\lambda\in\CC\}$
and all $N\in\NN$,
\begin{equation} \label{eq:nu+N}
\nu_N(T^+)\ =\ \min\left\{\min_{j=1}^w \nu_{j..j+N-1}(T^+)\ ,\ \nu_N(T|_\NN)\right\},
\end{equation}
where $\nu_{l..r}(T^+)$ is the smallest singular value of the rectangular submatrix
\begin{equation} \label{eq:SV1}
(T^+_{ij})_{i\in 1..r+w,\, j\in l..r},\qquad l,r\in\NN,
\end{equation}
of $T^+$.
If, additionally, \eqref{eq:CA=CA'} holds for $A$, then $\nu_N(T|_\NN)=\nu_N(T)$, so that
\begin{equation} \label{eq:nu+}
\nu_N(T^+)\ =\ \min\left\{\min_{j=1}^w \nu_{j..j+N-1}(T^+)\ ,\ \nu_N(T)\right\}.
\end{equation}
In particular, by \eqref{eq:localnu} and \eqref{eq:invnu}, 
\[
\nu(T^+)\ \le\ \nu(T)
\qquad\text{and}\qquad
\|(T^+)^{-1}\|\ 
\ge\ \|T^{-1}\|,
\]
so that,
\[
\sigma(A)\subseteq \sigma(A^+)
\qquad\text{and}\qquad
\sigma_\eps(A)\subseteq \sigma_\eps(A^+),\quad \eps>0.
\]
\end{proposition}

\begin{proof}
We fix $N\in\NN$, apply \eqref{eq:columns_cut} and \eqref{eq:columns_intact} to $T=A-\lambda$, and conclude \eqref{eq:nu+N} via Lemma \ref{lem:nuN-via-CN}:
\begin{align*}
\nu_N(T^+)\ &=\ \inf_{j\in\NN}\nu_{j..j+N-1}(T^+)\\
&=\ \min\left\{\min_{j=1}^w \nu_{j..j+N-1}(T^+)\ ,\ \inf_{j\ge w+1}\nu_{j..j+N-1}(T^+)\right\}\\
&\stackrel{\eqref{eq:columns_intact}}{=}\ \min\left\{\min_{j=1}^w \nu_{j..j+N-1}(T^+)\ ,\ \inf_{j\ge w+1}\nu_{j..j+N-1}(T)\right\}\\
&\stackrel{\eqref{eq:columns_cut}}{=}\ \min\left\{\min_{j=1}^w\Big\{\nu_{j..j+N-1}(T^+),\ \nu_{j..j+N-1}(T)\Big\}\ ,\ \inf_{j\ge w+1}\nu_{j..j+N-1}(T)\right\}\\
&=\ \min\left\{\min_{j=1}^w \nu_{j..j+N-1}(T^+)\ ,\ \inf_{j\in\NN}\nu_{j..j+N-1}(T)\right\}\\
&=\ \min\left\{\min_{j=1}^w \nu_{j..j+N-1}(T^+)\ ,\ \nu_N(T|_\NN)\right\}.
\end{align*}
Since also \eqref{eq:CA=CA'} transfers from $A$ to all $A-\lambda$, we get $\nu_N(T)=\nu_N(T|_\NN)$ and hence \eqref{eq:nu+}, by Proposition \ref{prop:NCC}.
The rest is by Lemma~\ref{lem:auto_selfcontained} and Lemma~\ref{lem:self2}.
\end{proof}
\begin{remark} \label{rem:nuT+}
Letting $N\to\infty$ in \eqref{eq:nu+N} yields $\nu(T^+) = \min\{c,\nu(T|_\NN)\}$, where
\[
c\ :=\ \lim_{N\to\infty} \min_{j=1}^w \nu_{j..j+N-1}(T^+)\ =\ \min_{j=1}^w\nu_{j..}(T^+)\ =\ \nu_{1..}(T^+)\ =\ \nu(T^+)
\] 
since $\lim$ and $\min$ commute and since $\nu_{l..r}(T^+)\to\nu_{l..}(T^+)$ as $r\to\infty$ (by monotonicity and Corollary \ref{cor:nuN}).
This equality shows that the second term in the minimum of \eqref{eq:nu+N} and \eqref{eq:nu+} is asymptotically irrelevant.
\end{remark}

Now we apply Proposition \ref{prop:band+} to generalized Schrödinger operators $A=H(b)$:
\begin{corollary}
Let $b\in\S^\ZZ$. If $\W(b)=\W(b|_\NN)$ then, for all $\lambda\in\CC$,
\[
\|(H(b)^+-\lambda)^{-1}\|\ge\|(H(b)-\lambda)^{-1}\|,\quad\text{so that}\quad\sigma(H(b)^+)\supset\sigma(H(b)),
\]
where \eqref{eq:nu+} applies to $T=H(b)-\lambda$.
\end{corollary}
\begin{proof}
Note \eqref{eq:CA=CA'} for $A=H(b)$, by $\W(b)=\W(b|_\NN)$, and use Proposition~\ref{prop:band+}.
\end{proof}

\subsection{Two Schrödinger operators on the half axis}
\begin{proposition} \label{prop:nu1sided}
Let $b,c\in\S^\NN$ and let $H(b)^+$ and $H(c)^+$ denote the corresponding generalized Schrödinger operators on the half axis. Denote their band-width by $w$. If $N\in\NN$ and
\begin{equation} \label{eq:2nu_N+}
b|_{1..w+N-1}\ =\ c|_{1..w+N-1}
\qquad\text{and}\qquad
\W_N(b)\ \subseteq\ \W_N(c)
\end{equation}
then
\[
\nu_N(H(b)^+)\ \ge\ \nu_N(H(c)^+)
\]
\end{proposition}
\begin{proof}
Let $N\in\NN$ and note that
\begin{align*}
\nu_N(H(b)^+) \ &\stackrel{\eqref{eq:nu+N}}=\ \min\left\{\min_{j=1}^w \nu_{j..j+N-1}(H(b)^+)\ ,\ \nu_N(H(b)|_\NN)\right\}\\
&\ \ge\ \min\left\{\min_{j=1}^w \nu_{j..j+N-1}(H(c)^+)\ ,\ \nu_N(H(c)|_\NN)\right\}\ \stackrel{\eqref{eq:nu+N}}=\ \nu_N(H(c)^+)
\end{align*}
since the two $\min_{j=1}^w \nu_{j..j+N-1}\dots$ terms are equal, by 
$b|_{1..w+N-1}\ =\ c|_{1..w+N-1}$ and 
$\nu_N(H(b)|_\NN)\ge\nu_N(H(c)|_\NN)$, by $\W_N(b)\subseteq\W_N(c)$ and Proposition~\ref{prop:NCC}~(a).
\end{proof}
For a direct comparison of $\nu(H(b)^+)$ and $\nu(H(c)^+)$ in the same style, 
we have to send $N\to\infty$, enforcing $b=c$ via $b|_{1..w+N-1}=c|_{1..w+N-1}$ for all $N\in\NN$. 
For an approximation $\nu(H(b_m)^+)\to\nu(H(b)^+)$ however, we have much less restrictive conditions (including pointwise convergence $b_m\to b$), see Theorem \ref{thm:approxspec+} below.

\subsection{Spectral approximation via approximation of submatrices}
We start with the general case, not necessarily generalized Schrödinger. In particular, many diagonals could be varying.
\begin{proposition} \label{prop:approxspec}
Let $A,A_1,A_2,\ldots$ be band operators on the axis with a uniform upper bound on their norms and on their band-widths.
If
\begin{equation} \label{eq:ev_equalsubwordsC}
\forall N\in\NN\ \exists m_0\in\NN:\ \forall m\ge m_0:\quad \C_N(A_m)\ =\ \C_N(A)
\end{equation}
then $\nu(A_m)\to\nu(A)$ and, in fact,
\[
\nu(A_m-\lambda)\ \to\ \nu(A-\lambda),\qquad\lambda\in\CC.
\]
\end{proposition}
\begin{proof}
Let $\eps>0$ and take $N\in\NN$ so that Lemma~\ref{lem:nuN} applies, with $\frac{\eps}2$ in place of $\eps$, to $A$ and all $A_m$ with $m\in\NN$. This is possible since we have a uniform upper bound on their norms and on their band-widths.

Now, in accordance with $N$, take $m$ large enough that \eqref{eq:ev_equalsubwordsC} holds. 
Then, by Lemma~\ref{lem:nuN}, Proposition~\ref{prop:NCC}~(a) and again Lemma~\ref{lem:nuN}, in this order,
\[
\nu(A_m)\stackrel{\eps/2}\approx\nu_N(A_m)=\nu_N(A)\stackrel{\eps/2}\approx\nu(A),
\quad\text{so that}\quad
\nu(A_m)\stackrel\eps\approx\nu(A).
\]
Since $\eps$ was arbitrary, it follows that $\nu(A_m)\to\nu(A)$ as $m\to\infty$.
Repeating the same argument for $A_m-\lambda$ and $A-\lambda$ in place of $A_m$ and $A$ proves the claim.
\end{proof}

\subsection{Spectral approximation via approximation of subwords: the axis}
We directly conclude the same (and more) for a sequence of generalized Schrödinger operators on the axis.
Note that a similar result is already known in the self-adjoint case~\cite{Beckus.2020}.
\begin{theorem} \label{thm:approxspec}
Let $b,b_1,b_2,\ldots\in\S^\ZZ$ and look at the corresponding generalized Schrödinger operators on the axis, $A:=H(b)$ as well as $A_m:=H(b_m)$ for $m\in\NN$.
If
\begin{equation} \label{eq:ev_equalsubwords}
\forall N\in\NN\ \exists m_0\in\NN:\ \forall m\ge m_0:\quad \W_N(b_m)\ =\ \W_N(b)
\end{equation}
then, for all $\lambda\in\CC$ and $\eps>0$,
\[
\|(A_m-\lambda)^{-1}\|\ \to\ \|(A-\lambda)^{-1}\|
\qquad\text{and}\qquad
\sigma_\eps(A_m)\to\sigma_\eps(A)
\]
in Hausdorff distance, as $m\to\infty$.
If $A_m$ and $A$ are normal, one also has
\[
\sigma(A_m)\to\sigma(A).
\]
\end{theorem}
\begin{proof}
We apply Proposition \ref{prop:approxspec}, which is possible since all our operators have the same band-width and $\| A\|, \|A_m\|\le \|L\|+\max\limits_{\sigma\in\S} |\sigma|$ for all $m\in\NN$. 

With the same argument for the adjoints of our operators, we get, after recalling~\eqref{eq:invnu}, that $\|(A_m-\lambda)^{-1}\| \to \|(A-\lambda)^{-1}\|$ for all $\lambda\in\CC$.
Concluding Hausdorff convergence of the pseudospectra from here is a standard result
(e.g.~\cite[Section 2.3]{Colb:PE_pBC} or \cite{LiSchmeck:Haus}).
In general, one cannot conclude Hausdorff convergence of the spectra. But in the normal case
this is possible, due to \eqref{eq:nu_distspec}.
\end{proof}

\begin{remark}
For pseudoergodic $b\in\S^\ZZ$, it is well-known \cite{Davies2001:PseudoErg} that
\begin{equation} \label{eq:spec_PE}
\sigma(H(b))
\ =\ \bigcup_{c\in\S^\ZZ} \sigma(H(c)),
\end{equation}
and an important question, for example in \cite{MartinezThesis,LiBiDiag,CWDavies2011,HaggerPhD}, is whether restricting the union on the right to periodic $c\in\S^\ZZ$ yields a dense subset.
We are not quite answering this one, in large generality, but a related question:

In analogy to \eqref{eq:spec_PE}, one can prove (e.g. \cite{Chandler-WildeLindner.2016})
\begin{equation} \label{eq:epsspec_PE}
\sigma_\eps(H(b))\ =\ \bigcup_{c\in\S^\ZZ} \sigma_\eps(H(c))
\ \supset\!\!\! \bigcup_{\text{periodic }c\in\S^\ZZ} \sigma_\eps(H(c)),\qquad\eps>0.
\end{equation}
Our construction, approximating $\sigma_\eps(H(b))$ in Hausdorff distance by $\eps$-pseudospectra of $H(b_m)$ with periodic $b_m\in\S^\ZZ$, shows that the subset on the right of \eqref{eq:epsspec_PE} is dense in the set on the left, $\sigma_\eps(H(b))$. If $H(b_m)$ and $H(b)$ are normal then also the original question about spectra and \eqref{eq:spec_PE} is answered affirmatively.
\end{remark} 

Note that Theorem~\ref{thm:approxspec}, in particular~\eqref{eq:ev_equalsubwords}, does not require any convergence $b_m\to b$. This changes if we switch to the half-axis.

\subsection{Spectral approximation via approximation of subwords: the half-axis}
Here is the corresponding result for the half-axis, for the self-adjoint case see~\cite{Kellendonk.2019}.
\begin{theorem} \label{thm:approxspec+}
Let $\S\subseteq\CC$ be finite,
let $b,b_1,b_2,\ldots\in\S^\NN$ and look at the corresponding half-axis generalized Schrödinger operators $A^+ \coloneqq H(b)^+$ as well as $A_m^+ \coloneqq H(b_m)^+$ for $m\in\NN$.

If, again, \eqref{eq:ev_equalsubwords} holds and, in addition, $b_m\to b$ pointwise,
then, for all $\lambda\in\CC$ and $\eps>0$,
\[
\|(A_m^+-\lambda)^{-1}\|\ \to\ \|(A^+-\lambda)^{-1}\|
\qquad\text{and}\qquad
\sigma_\eps(A_m^+)\to\sigma_\eps(A^+)
\]
in Hausdorff distance, as $m\to\infty$. If $A_m^+$ and $A^+$ are normal, one also has
\[
\sigma(A_m^+)\to\sigma(A^+).
\]
\end{theorem}
\begin{proof}
The proof is almost identical to that of Theorem~\ref{thm:approxspec}: 
again, given $\eps>0$, choose $N\in\NN$ so that Lemma \ref{lem:nuN} holds with $\frac\eps2$ for $A^+$ and all $A_m^+$. 
Then, and this is a bit different, in accordance with $N$, take $m_0$ large enough that
\begin{itemize}
\item \eqref{eq:ev_equalsubwords} holds, as well as 
\item $b_m=b$ on $1..w+N-1$ for all $m\ge m_0$, 
\end{itemize}
the latter is possible since $b_m\to b$ pointwise and $\S$ is discrete.

Then, by Proposition~\ref{prop:nu1sided}, $\nu_N(A_m^+)=\nu_N(A^+)$, and we proceed as in the proof of Proposition \ref{prop:approxspec}.
\end{proof}

\subsection{Sufficient conditions for (\ref{eq:ev_equalsubwords})}
Condition \eqref{eq:ev_equalsubwords}, in short: 
\begin{center}
for all $N\in\NN$, eventually, as $m\to\infty$, $\W_N(b_m)=\W_N(b)$, 
\end{center}
plays a crucial role in Proposition \ref{prop:approxspec} and Theorems \ref{thm:approxspec} and \ref{thm:approxspec+}.
Here we discuss some arguably more handy criteria that are sufficient for \eqref{eq:ev_equalsubwords}. 
We look at the case of the whole axis but the half-axis case works the same way (with obvious modifications, like replacing $-r..r$ by $1..r$).

For aperiodic potentials and potentials generated by substitutions,
this condition can always be satisfied with periodic potentials $b_m$ that can be obtained constructively.
We will leave the details to~\cite{Beckus.2020} and~\cite{Gabel.2021}.

In the general case, we discuss the inclusions 
$\W_N(b_m)\supseteq\W_N(b)$ and $\W_N(b_m)\subseteq\W_N(b)$ separately because each is interesting in its own right: ``$\supseteq$'' guarantees a spectral inclusion of $\sigma(H(b))$, and ``$\subseteq$'' at least avoids spectral pollution.

For both inclusions, we make the following assumptions:
\begin{enumerate}[label=\arabic*.)]
\item \label{it:assumption1} $|\S|<\infty$,
\item [~] A bounded $\S\subseteq\CC$ is discrete if and only if it is finite (Bolzano-Weierstrass).
\item  \label{it:assumption2} $b_m\in\S^\ZZ$ for all $m\in\NN$,
\item [~] Together with 1.) and 3.) we also conclude $b\in\S^\ZZ$. 
\item  \label{it:assumption3} $b_m\to b$ pointwise.
\end{enumerate}
Note: In the half-axis case (Theorem \ref{thm:approxspec+}), we already make all three assumptions.
By~\ref{it:assumption1}~and~\ref{it:assumption3}, we have that
\begin{equation} \label{eq:agreelocally}
\forall r\in\NN\ \ \exists m(r):\ \forall m\ge m(r): \quad b_m|_{-r..r}=b|_{-r..r}\,.
\end{equation}

\subsection*{Sufficient conditions for $\W_N(b_m)\supseteq\W_N(b)$, eventually}
This inclusion holds without further conditions; 1.) -- 3.) are enough. 
Indeed:
 Let $N\in\NN$ and $w\in\W_N(b)$.
 Take $r\in\NN$ large enough for $w\in\W(b|_{-r..r})$.
 For $m\ge m_0:=m(r)$ from \eqref{eq:agreelocally} we then have $b_m|_{-r..r}=b|_{-r..r}$, whence $w\in\W_N(b_m)$.

\subsection*{Sufficient conditions for $\W_N(b_m)\subseteq\W_N(b)$, eventually}
This property is not for free.
As a negative example, look at $b_m:=\chi_{\{m\}}\to b\equiv 0$, so that subwords of $b_m$ containing ``$1$'' 
do not show in the limit $b$.
To rule out these kinds of examples, we will impose that each pattern that appears once in $b_m$ appears infinitely often (this is not enough yet, e.g.~$b_m=\chi_{2m\ZZ+m}\to b\equiv 0$) and that the gap between two occurrences remains bounded as $m\to\infty$. (see e.g.~\cite{DamanikLenz2006} and Remark 4.5 in \cite{Li:MinLimOps}).
 
For $u\in\S^\ZZ$, $N\in\NN$ and $w\in\W_N(u)$, recall 
\[
\mathrm{pos}(w,u)\ =\ \{k\in\ZZ: u(k)\cdots u(k+N-1)=w\}
\]
and further put
\begin{align*}
	\mathrm{gap}(w,u)&\ :=\ \min\{r\in\NN: \mathrm{pos}(w,u)+(-r..r)=\ZZ\},\\
	\mathrm{gap}(N,u)&\ :=\ \max\{\mathrm{gap}(w,u):w\in\W_N(u)\}.
\end{align*}
If $\mathrm{gap}(w,u)<\infty$,  we say that \emph{$w$ has bounded gaps in $u$}. 
Then $\mathrm{gap}(N,u)<\infty$ means that every subword of length $N$ has bounded gaps in $u$. (By 1., there are only finitely many $w\in\W_N(u)$, so that a uniform upper bound is automatic.)

The next step is to expect this property for all $b_m$, uniformly in $m$:
\begin{proposition}\label{prop:bounded_gaps}  
In addition to~\ref{it:assumption1}~--~\ref{it:assumption3}, assume
\begin{enumerate}[start=4, label=\arabic*.)]
\item\label{it:assumption4} $\quad\forall N\in\NN:\ \ g(N):=\sup\limits_{m\in\NN}\ \mathrm{gap}(N,b_m)<\infty$.
\end{enumerate}

Then, for all $N\in\NN$, $\W_N(b_m)\subseteq\W_N(b)$ holds for all sufficiently large $m$.

\end{proposition}\label{prop:bounded_gaps}
\begin{proof}
Let $N\in\NN$, $r\in\NN$ with $2r>g(N)$ from~\ref{it:assumption4}
and then take $m\ge m_0:=m(r)$ from~\eqref{eq:agreelocally}.
Now let $w\in\W_N(b_m)$.
By~\ref{it:assumption4}, we have 
\[
	\mathrm{gap}(w,b_m)\ \le\ \mathrm{gap}(N,b_m)\ \le\ g(N)\ <\ 2r.
\]
So there exists a $k\in \mathrm{pos}(w,b_m)\cap-r..r$
and therefore an occurrence of $w$ in $b_m|_{-r..r}$
and, by \eqref{eq:agreelocally}, also in $b|_{-r..r}$.
So $w\in\W_N(b)$.
\end{proof}

While Proposition~\ref{prop:bounded_gaps} might help to decide whether some given $b$ and $(b_m)$ satisfy \eqref{eq:ev_equalsubwords}, it does not help to construct $(b_m)$ from $b$.
Let us therefore sketch a straightforward way to create a sequence $(b_m)$ such that \eqref{eq:ev_equalsubwords} is always satisfied.
For that purpose we to restrict ourselves to the case where $\W(b)=\W(b|_\NN)$, in which case, by Lemma \ref{lem:auto_selfcontained}, $\#(w,b|_\NN) = \infty$ for all $w\in\W(b)$.

Let $N\in \NN$ and find $r\in\NN$ such that $\W_N(b) = \W_N(b|_{1..r})$.
Then take $\widetilde{r} \coloneqq \min\{\mathrm{pos}(b|_{1..N},b|_{r+1..}) \}-1$ and set $b_N$ to be the $\widetilde{r}$-periodic word repeating $b|_{1..\widetilde{r}}$.
It is easy to see that \eqref{eq:ev_equalsubwords} is always satisfied with $m_0=N$.

\subsection{Directions of extension} \label{sec:ext}
We have seen (Proposition \ref{prop:approxspec}, one-way model in Section \ref{sec:ex_math_phys}) that the limitation to just one varying diagonal was merely for convenience. But there is more room for extensions:
Instead of band operators on $\ell^2(\ZZ)$, we could pass to uniform limits of band operators on $\ell^p(\ZZ^d,X)$, where $p\in[1,\infty]$, $d\in\NN$ and $X$ is any Banach space. 
\begin{itemize}
\item How to generalize Lemma \ref{lem:nuN} from the set of band operators to its closure, the so-called \emph{band-dominated operators}~\cite{RaRoSi:Book,LiBook}, 
is shown in~\cite{BigQuest}.
\item The generalization from $p=2$ to arbitrary $p\in[1, \infty]$ is possible using so-called ${\mathcal P}$-theory~\cite{Sei:Survey}, 
but it does not change the spectra, see Lemma~\ref{prop:kurbatov}, as long as we move away from band operators not too far. 
Precisely, it holds in the so-called \emph{Wiener algebra}~\cite{Li:Habil}, which is somewhere between the classes of band and band-dominated operators.
In particular, Lemma~\ref{lem:nuN} is not restricted to the case $p=2$.
\item Passing from scalar to $X$-valued sequence spaces is again covered by ${\mathcal P}$-theory \cite{Sei:Survey} 
and enables us to study, for example, $L^p(\RR)$, by identifying it with $\ell^p(\ZZ,L^p[0,1])$.
\item Our tools (e.g.~Lemma~\ref{lem:nuN}) and ideas are in fact not limited to the 1D case, $d=1$. The case $d\ge 2$ just needs some obvious modifications:
\begin{itemize}
\item discrete intervals have to be replaced by, e.g.~cartesian products of discrete intervals or, more flexibly, 
by bounded sets $S\subseteq\ZZ^d$ for which $S+[-\frac12,\frac12]^d$ is connected in $\RR^d$;
\item a finite subword of $b\in\S^{(\ZZ^d)}$ is then $b|_S$ with $S$ from above;
\item in the same style, a finite submatrix of consecutive columns is then $(A_{i,j})_{i\in\ZZ^d,\,j\in S}$, 
truncated to the band part of $A$ and shifted to a common region;
\item the full space results are again straightforward,
\item for compressions to an infinite set $U\subseteq\ZZ^d$ (like the half-axis in 1D), 
the new Proposition~\ref{prop:nu1sided} will again have to require exactness in the $w$-neighborhood of the boundary of $U$ and a match of the sets of finite subwords otherwise.
\item Note however that Lemma~\ref{lem:self2} does not transfer to $d\ge 2$ since the main results of~\cite{Sei:semi} rely on 1D. 
(Also see Example 30 in~\cite{Sei:Survey}.)
\end{itemize}
\end{itemize}

\section*{Acknowledgement}
The authors would like to thank Dennis Clemens, Pranshu Gupta, Alexander Haupt and Anusch Taraz for providing stimulating discussions about de Bruijn sequences and Sturmian words.


\begin{thebibliography}{[0]}
	

\bibitem{Anderson58}
P.~W.~Anderson,
Absence of diffusion in certain random lattices, 
{\it Phys. Rev.} {\bf 109} (1958), 1492--1505.

\bibitem{Bellissard.1989} 
J.~Bellissard, B.~Iochum, E.~Scoppola, and D.~Testard,
Spectral Properties of One Dimensional Quasi-Crystals,
\textit{Commun.\ Math.\ Phys.\@} {\bfseries 125}, (1989), 527--543.

\bibitem{Beckus.2020} 
S.~Beckus, J.~Bellissard, and G.~De Nittis,
Spectral Continuity for Aperiodic Quantum Systems II. Periodic Approximations in 1D,
\textit{J.\ Math.\ Phys.\@} {\bfseries 61}, (2020), 123505.

\bibitem{SCI}
J.~Ben-Artzi, A.~C.~Hansen, O.~Nevanlinna and M.~Seidel,
New barriers in complexity theory: On the Solvability Complexity Index and towers of algorithms,
{\it C. R. Acad. Sci. Paris, Ser. I} {\bf 353} (2015), 931--936.


\bibitem{Boegli2017}
S.~B\"ogli,
Convergence of sequences of linear operators and their spectra,
\textit{Int. Eq. Op. Th.} {\bf 88} (2017), 559--599.

\bibitem{Boegli2020}
S.~B\"ogli, M.~Marletta and C.~Tretter,
The essential numerical range for unbounded linear operators,
\textit{J.~Func.~Anal.} {\bf 279} (2020), 108509.

\bibitem{BoeSi2}
A.~B\"ottcher and B.~Silbermann,
{\it Introduction to Large Truncated Toeplitz Matrices},
Springer, Berlin, Heidelberg, 1999.

\bibitem{Boulton08}
L.~Boulton, P.~Lancaster, P.~Psarrakos, 
On pseudospectra of matrix polynomials and their boundaries,
{\it Mathematics of Computation}, 
{\bf 77} (2008), 313--334. 

\bibitem{BrezinZee}
E.~Brézin and A. Zee,
Non-Hermitean delocalization: multiple scattering and bounds,
{\it Nuclear Phys. B} {\bf 509} (1998), 599--614.

\bibitem{CWChML:3diag}
S.~N.~Chandler-Wilde, R. Chonchaiya and M. Lindner,
On the spectra and pseudospectra of a class of non-self-adjoint random matrices and operators,
{\it Oper.~Matrices}, {\bf 7} (2013), 739--775.

\bibitem{CW.Heng.ML:UpperBounds}
S.~N.~Chandler-Wilde, R.~Chonchaiya and M.~Lindner,
On Spectral Inclusion Sets and Computing the Spectra and Pseudospectra of Bounded Linear Operators,
\textit{in preparation}.

\bibitem{CWDavies2011}
S.~N.~Chandler-Wilde and E.~B.~Davies,
Spectrum of a {F}einberg-{Z}ee random hopping matrix,
{\it Journal of Spectral Theory}, {\bf 2} (2012), 147--179

\bibitem{CW.Hagger}
S.~N.~Chandler-Wilde and R. Hagger,
On Symmetries of the Feinberg-Zee Random Hopping Matrix,
{\it Operator Theory: Advances and Applications} {\bf 258} (2017), 51--78.

\bibitem{CWLi:Favard} 
S.~N.~Chandler-Wilde and M.~Lindner,
Sufficiency of Favard’s condition for a class of band-dominated operators on the axis,
\textit{J.\ Funct.\ Anal.\@}
\textbf{254} 
(2008),
1146--1159.

\bibitem{Chandler-WildeLindner.2011} 
S.~N.~Chandler-Wilde and M.~Lindner,
\textit{Limit Operators, Collective Compactness, and the Spectral Theory of Infinite Matrices},
American Mathematical Society, 2011.

\bibitem{Chandler-WildeLindner.2016} 
S.~N.~Chandler-Wilde and M.~Lindner,
Coburn's lemma and the finite section method for random Jacobi operators,
\textit{J.\ Funct.\ Anal.\@}
\textbf{270} 
(2016),
802--841.

\bibitem{CicutaContediniMolinari2000} 
G.~M.~Cicuta, M.~Contedini and L.~Molinari,
Non-Hermitian tridiagonal random matrices and returns to the origin of a random walk,
{\em J. Stat. Phys.} {\bf 98} (2000), 685--699.

\bibitem{Colb:PE_pBC}
M.~J.~Colbrook,
Pseudoergodic operators and periodic boundary conditions,
{\it Mathematics of Computation} {\bf 89} (2020), 737--766.

\bibitem{ColbRomanHansen:PRL}
M.~J.~Colbrook, B.~Roman and A.~C.~Hansen,
How to compute spectra with error control,
{\it Phys. Review Lett.} {\bf 122} (2019), 250201.


\bibitem{CovenHedlund}
E.~M.~Coven and G.~A.~Hedlund,
Sequences with minimal block growth, 
{\it Theory of Computing Systems} {\bf 7} (1973), 138--153.



\bibitem{Damanik.2015} 
D.~Damanik, M.~Embree, and A.~Gorodetski,
Spectral properties of Schrödinger operators arising in the study of quasicrystals,
in J.~Kellendonk, D.~Lenz, and J.~Savinien, eds.,
\textit{Mathematics of Aperiodic Order}, Progress in Mathematics, 
Vol.\ 309, Birkhäuser, 2015, pp.\ 307--370.

\bibitem{Damanik.2016} 
D.~Damanik, A.~Gorodetski, and W.~Yessen,
The Fibonacci Hamiltonian,
\textit{Invent.\ math.\@}
\textbf{206}
(2016),
629--692. 



\bibitem{DamanikLenz2006}
D.~Damanik and D.~Lenz,
Substitution dynamical systems: Characterization of linear repetitivity and applications,
{\it J. Math. Anal. Appl.} {\bf 321} (2006), 766--780.


\bibitem{Davies2001:SpecNSA}
E.~B.~Davies,
Spectral properties of non-self-adjoint matrices and operators, 
{\it Proc. Royal Soc. A.} {\bf 457} (2001), 191--206.

\bibitem{Davies2001:PseudoErg}
E.~B.~Davies, 
Spectral theory of pseudo-ergodic operators,
{\it Commun. Math. Phys.} {\bf 216} (2001), 687--704.

\bibitem{Davies.2007} 
E.~B.~Davies,
\textit{Linear Operators and their Spectra}, 
{Cambridge Univ. Press}, 
2007.

\bibitem{deBruijn}
N.G. de~Bruijn,
A combinatorial problem,
\textit{Indag. Math.}, {\bf 8} (1946), 461--467.

\bibitem{FeinZee97}
J.~Feinberg and A.~Zee,
Non-Hermitean Localization and De-Localization, 
{\it Phys. Rev. E} {\bf 59} (1999), 6433--6443.

\bibitem{Jacob21}
A.~Frommer, B.~Jacob, L.~ Vorberg, C.~Wyss, I.~ Zwaan,
Pseudospectrum enclosures by discretization,
\textit{Int. Eq. Op. Th.} {\bf 93}, 9 (2021). 

\bibitem{Gabel.2023} 
F.~Gabel, D.~Gallaun, J.~Großmann, M.~Lindner, and R.~Ukena,
\textit{Example Potentials for Spectral Approximation of Generalized Schrödinger Operators via Approximation of Subwords}, 
TUHH Universitätsbibliothek, 2023. 
Available at \url{https://doi.org/10.15480/336.4846}.

\bibitem{Gabel.2021} 
F.~Gabel, J.~Gro{\ss}mann, D.~Gallaun, M.~Lindner, R.~Ukena,
\textit{Finite section method for aperiodic Schrödinger operators},
preprint,
2021.
Available at \url{https://arxiv.org/abs/2104.00711}.

\bibitem{GoldKoru}
I.~Goldsheid and B.~Khoruzhenko,
Eigenvalue curves of asymmetric tridiagonal random matrices, 
\textit{Electronic Journal of Probability} {\bf 5} (2000), 1--28.

\bibitem{Hagen.2001} 
R.~Hagen, S.~Roch, and B.~Silbermann,
\textit{C*-Algebras and Numerical Analysis}, 
{CRC Press}, 
2000.

\bibitem{Hagger:dense}
R.~Hagger, 
The eigenvalues of tridiagonal sign matrices are dense in the spectra of periodic tridiagonal sign operators. 
\textit{Journal of Functional Analysis} {\bf 269} (2015), 1563--1570.

\bibitem{Hagger:symm}
R.~Hagger,
Symmetries of the Feinberg-Zee random hopping matrix, 
\textit{Random Matrices: Theory and Applications} {\bf 4} (2015).

\bibitem{HaggerPhD} 
R.~Hagger,
\textit{Fredholm Theory with Applications to Random Operators},
Ph.D. thesis, Technische Universität Hamburg, 2016. 
Available at \url{https://doi.org/10.15480/882.1272}.

\bibitem{HagLiSei} 
R.~Hagger, M.~Lindner, and M.~Seidel,
Essential pseudospectra and essential norms of band-dominated operators,
\textit{J.\ Math.\ Anal.\ Appl.\@}
\textbf{437}
(2016),
255--291.

\bibitem{Hansen:nPseudo}
A.~C.~Hansen,
On the solvability complexity index, the $n$-pseudospectrum and approximations of spectra of operators,
{\it J. Amer. Math. Soc.} {\bf 24} (2011), 81--124.

\bibitem{HatanoNelson1997}
N.~Hatano and D.~R.~Nelson,
Vortex pinning and non-Hermitian quantum mechanics,
{\em Phys. Rev. B} {\bf 56} (1997), 8651--8673.

\bibitem{HolzOrlZee}
D.E.~Holz, H.~Orland and A.~Zee,
On the remarkable spectrum of a non-Hermitian random matrix model,
{\it J. Phys. A Math. Gen.} {\bf 36} (2003), 3385--3400.

\bibitem{Kellendonk.2019}
Kellendonk, J., Prodan, 
E. Bulk–Boundary Correspondence for Sturmian Kohmoto-Like Models. 
{\it Ann. Henri Poincaré} {\bf 20} (2019), 2039–-2070. 

\bibitem{Kurbatov.1999} 
V.~G.~Kurbatov,
\textit{Functional Differential Operators and Equations},
Springer, 1999.

\bibitem{LiBook} 
M.~Lindner,
\textit{Infinite Matrices and their Finite Sections: An Introduction to the Limit Operator Method},
Birkhäuser, 2006.

\bibitem{LiBiDiag}
M.~Lindner,
A note on the spectrum of bi-infinite bi-diagonal random matrices,
{\em Journal of Analysis and Applications} {\bf 7} (2009), 269--278.

\bibitem{Li:Habil}
M.~Lindner,
\textit{Fredholm Theory and Stable Approximation of Band Operators and Generalisations}, 
Habilitation thesis, Technische Universität Chemnitz, 2009.
Available at \url{https://nbn-resolving.org/urn:nbn:de:bsz:ch1-200901182}.

\bibitem{Li:MinLimOps}
M.~Lindner,
Minimal families of limit operators,
{\it Operators and Matrices}
{\bf 16} (2022), 529--543.

\bibitem{LindnerRoch2010}
M.~Lindner and S.~Roch,
Finite sections of random Jacobi operators,
{\it SIAM J. Numer. Anal.} {\bf 50} (2012), 287--306.

\bibitem{LiSchmeck:Haus}
M.~Lindner and D.~Schmeckpeper,
A note on Hausdorff-convergence of pseudospectra,
{\it Opuscula Math.}
\textbf{41}
(2023),
101--108.


\bibitem{BigQuest} 
M.~Lindner and M.~Seidel,
An affirmative answer to a core issue on limit operators,
\textit{J.\ Funct.\ Anal.\@}
\textbf{267} 
(2014),
901--917.

\bibitem{LiSchmidt:Givens}
M.~Lindner and T.~Schmidt,
Recycling givens rotations for the efficient approximation of pseudospectra of band-dominated operators,
{\it Operators and Matrices}, {\bf 11} (2017), 1171--1196.
  

\bibitem{MartinezThesis}
C.~Mart\'inez,
Spectral Properties of Tridiagonal Operators,
{\it PhD thesis, Kings College, London} 2005.


\bibitem{NelsonShnerb}
D.R.~Nelson and N.M.~Shnerb,
Non-Hermitian localization and population biology, 
{\it Phys. Rev. E} {\bf 58} (1998),
1383--1403.

\bibitem{RaRoSi:Book} 
V.~S.~Rabinovich, S.~Roch, and B.~Silbermann,
\textit{Limit operators and their applications in operator theory}, {Birkh{\"a}user}, 
2004.

\bibitem{Seidel:Neps}
M.~Seidel,
On $(N,\eps)$-pseudospectra of operators on Banach spaces,
{\it Journal of Functional Analysis}, {\bf 262} (2012), 4916--4927.

\bibitem{Sei:Survey}
M. Seidel,
Fredholm theory for band-dominated and related operators: a survey,
{\it Linear Algebra Appl.} \textbf{445} (2014), 373--394.

\bibitem{Sei:semi}
M.~Seidel,
On Semi-Fredholm Band-Dominated Operators,
{\it Int.~Eq.~Op.~Th.} {\bf 83} (2015), 35--47.


\bibitem{Shechtman.1984} 
D.~Shechtman, I.~Blech, D.~Gratias, and J.~Cahn,
Metallic phase with long-range orientational order and no translational symmetry, 
\textit{Phys.\ Rev.\ Lett.\@}
\textbf{53}, 
(1984),
1951--1953.

\bibitem{Stollmann}
P. Stollmann, 
{\it Caught by Disorder. Bound states in random media,}
Progress in Mathematical Physics, 20, Birkh\"auser, Boston, 2001.

\bibitem{Suto.1987} A.~S\"ut\H{o},
The spectrum of a quasiperiodic Schrödinger operator,
\textit{Commun.\ Math.\ Phys.\@}
\textbf{111}
(1987),
409--415.


\bibitem{TrefContEmb}
L.~N.~Trefethen, M.~Contedini and M.~Embree,
Spectra, pseudospectra, and localization for random bidiagonal matrices, 
{\it Comm. Pure Appl. Math.} {\bf 54} (2001), 595--623.

\bibitem{TrefEmb}
L.~N.~Trefethen and M.~Embree,
{\it Spectra and Pseudospectra:
the Behavior of Nonnormal Matrices and Operators}, Princeton Univ. Press,
Princeton, NJ, 2005.

\bibitem{Weber.2022} 
L.~Weber,
\textit{Random Walks and Tridiagonal Matrices},
Masters thesis, Technische Universität Hamburg, 2022. 




\end{thebibliography}
\end{document}